\documentclass[a4paper]{amsart}
\date{February 15, 2019}
\usepackage[dvipdfmx]{graphicx}
\usepackage{latexsym}
\usepackage{amssymb}
\usepackage{amsmath}
\usepackage{amscd}
\usepackage{amsthm}
\usepackage{amsfonts}
\usepackage{mathrsfs}
\usepackage{enumerate}
\usepackage{enumitem}
\usepackage{bm}
\usepackage[usenames]{color}
\usepackage[active]{srcltx}
\title{Geometric invariants of $5/2$-cuspidal edges}
\author[A.~Honda]{Atsufumi Honda}
\address{%
   Department of Applied Mathematics, 
   Faculty of Engineering, Yokohama National University, 
   79-5 Tokiwadai, Hodogaya, Yokohama 240-8501, Japan
}
\email{honda-atsufumi-kp@ynu.jp}
\author[K.~Saji]{Kentaro Saji}
\address{%
   Department of Mathematics, 
   Kobe University,
   Rokko 1-1, Nada, 
   Kobe 657-8501, Japan
}
\email{saji@math.kobe-u.ac.jp}
\thanks{This work was supported by 
JSPS KAKENHI Grants numbered JP26400087, 16K17605.}
\dedicatory{
{Dedicated to Professor Takashi Nishimura on the occasion of 
his 60th birthday.}}
\usepackage{amsthm}
\theoremstyle{plain}
 \newtheorem{theorem}{Theorem}[section]

 \newtheorem{proposition}[theorem]{Proposition}
 \newtheorem{fact}[theorem]{Fact}
 \newtheorem*{fact*}{Fact}
 \newtheorem{lemma}[theorem]{Lemma}
 \newtheorem{corollary}[theorem]{Corollary}
 \theoremstyle{remark}
 \newtheorem{definition}[theorem]{Definition}
 \newtheorem{remark}[theorem]{Remark}
 \newtheorem*{acknowledgements}{Acknowledgements}
 \newtheorem{example}[theorem]{Example}
\numberwithin{equation}{section}

 \makeatletter
    
    \@addtoreset{equation}{section}
  \makeatother
\newcommand{\vect}[1]{\boldsymbol{#1}}
\newcommand{\e}{\vect{e}}
\newcommand{\bb}{\vect{b}}
\newcommand{\n}{\vect{n}}

\newcommand{\inner}[2]{\left\langle{#1},{#2}\right\rangle}

\newcommand{\ep}{\varepsilon}
\newcommand{\R}{\boldsymbol{R}}
\newcommand{\A}{\mathcal{A}}

\newcommand{\sgn}{\operatorname{sgn}}
\newcommand{\rank}{\operatorname{rank}}
\newcommand{\aaa}{\alpha_1}
\newcommand{\ab}{\alpha_2}
\newcommand{\ac}{\alpha_3}
\newcommand{\ad}{\alpha_4}
\begin{document}
\begin{abstract}
We introduce two invariants
called 
the \emph{secondary cuspidal curvature}
and the \emph{bias}
on $5/2$-cuspidal edges,
and investigate their basic properties.
While the secondary cuspidal curvature 
is an analog of
the cuspidal curvature of (ordinary) cuspidal edges, 
there are no invariants corresponding to the bias.
We prove that the product 
(called the \emph{secondary product curvature\/}) of 
the secondary cuspidal curvature 
and the limiting normal curvature
is an intrinsic invariant.
Using this intrinsicity, 
we show that any real analytic $5/2$-cuspidal edges 
with non-vanishing limiting normal curvature
admit non-trivial isometric deformations,
which provides the extrinsicity of various invariants.
\end{abstract}
\maketitle

\section{Introduction}
{In this paper,
we study local differential geometric properties
of curves and surfaces with singular points.
Since we look at local properties,
we essentially deal with map-germs: 
$(\R,0)\to(\R^2,0)$ and $(\R^2,0)\to(\R^3,0)$.
We consider invariants under an action that is a 
diffeomorphism on the source space and
an orientation preserving isometry 
of the target Euclidean space: $\R^2$ and $\R^3$.
In the case of curves,
we take a representative and 
identify a map-germ $(\R,0)\to(\R^2,0)$ with a curve $I\to\R^2$, 
where $I$ is an open interval including the origin of $\R$.
Similarly, in the case of surfaces,
we take a representative and 
identify a map-germ $(\R^2,0)\to(\R^3,0)$ with a surface $U\to\R^3$, 
where $U$ is an open neighborhood of the origin in $\R^2$.
We mainly deal with $5/2$-cusps and $5/2$-cuspidal edges
in this paper.}

The {\it ordinary cusp\/} or {\it $3/2$-cusp\/} is a map-germ
$(\R,0)\to(\R^2,0)$ which is
diffeomorphic
($\A$-equivalent) to the map-germ $t\mapsto(t^2,t^3)$
at the origin.
It is known that the $3/2$-cusp is the most frequently appearing
singularity on plane curves.
A {\it cuspidal edge\/} is a map-germ
$(\R^2,0)\to(\R^3,0)$ which is
$\A$-equivalent to the 
map-germ $(u,v)\mapsto(u,v^2,v^3)$
at the origin (Figure \ref{fig:CEor}, right),
where two map-germs $f,g:(\R^m,0)\to(\R^n,0)$ are
{\it $\A$-equivalent\/} if there exist diffeomorphisms
$\phi_s:(\R^m,0)\to(\R^m,0)$ and
$\phi_t:(\R^n,0)\to(\R^n,0)$ 
such that $\phi_t\circ f\circ\phi_s=g$.
{By definition, the image of a cuspidal edge
is diffeomorphic to a direct product
of a $3/2$-cusp with an interval ($\{(x,y,z)\,|\,y^3-z^2=0\}$)}, 
and its differential geometric properties
are well studied.
In \cite{u} (see also \cite{SUY2}), 
the 
{\it cuspidal curvature\/} for
$3/2$-cusps is defined.
Roughly speaking, 
the cuspidal curvature measures {whether
a $3/2$-cusp is narrower or wider}.
For cuspidal edges,
the {\it singular curvature\/} and
the {\it limiting normal curvature\/} are 
introduced in \cite{SUY}, and their geometric meanings are studied.

A {\it ${5/2}$-cusp\/} (respectively,
{\it ${5/2}$-cuspidal edge\/}) is a map-germ
$(\R,0)\to(\R^2,0)$ (respectively,
$(\R^2,0)\to(\R^3,0)$) which is
$\A$-equivalent to the 
map-germ $t\mapsto(t^2,t^5)$
(respectively,
$(u,v)\mapsto(u,v^2,v^5)$) 
at the origin (Figure \ref{fig:CEor}, left).
A ${5/2}$-cusp is also called a {\it rhamphoid cusp}.
{Although ${5/2}$-cuspidal edges do not generically
appear, it has been pointed out that}
they naturally appear in various differential
geometric situations \cite{HKS,ishiyama,porteous}.
For ${5/2}$-cusps,
the cuspidal curvature vanishes.
Hence, to measure the width of ${5/2}$-cusps,
we need to consider {higher order invariants}.
In this paper, we define two curvatures
on ${5/2}$-cusps in addition to the invariants
we mentioned above, which are
the 
{\it secondary cuspidal curvature\/}
and 
the 
{\it bias} of cusps.
The secondary cuspidal curvature is an analog
of cuspidal curvature of $3/2$-cusps, but
as we will see in Section \ref{sec:rhamphoidcusp},
there is no corresponding notion of
bias for $3/2$-cusps.
Using {these invariants, 
the secondary cuspidal curvature and 
the bias,}
we also define two curvatures for
${5/2}$-cuspidal edges.

\begin{figure}[htb]
\begin{center}
 \begin{tabular}{{c@{\hspace{15mm}}c}}
  \resizebox{5.2cm}{!}{\includegraphics{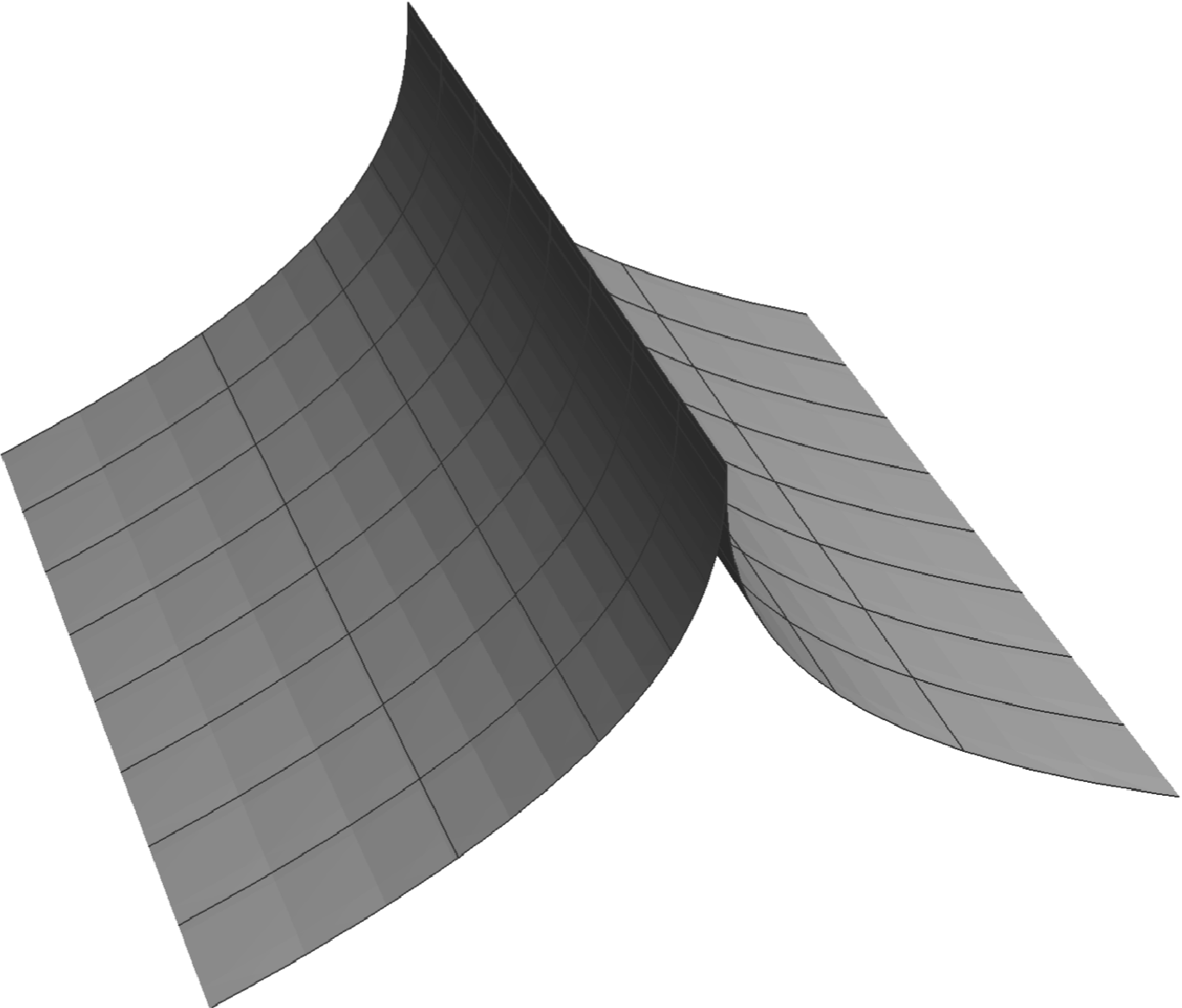}} &
  \resizebox{5cm}{!}{\includegraphics{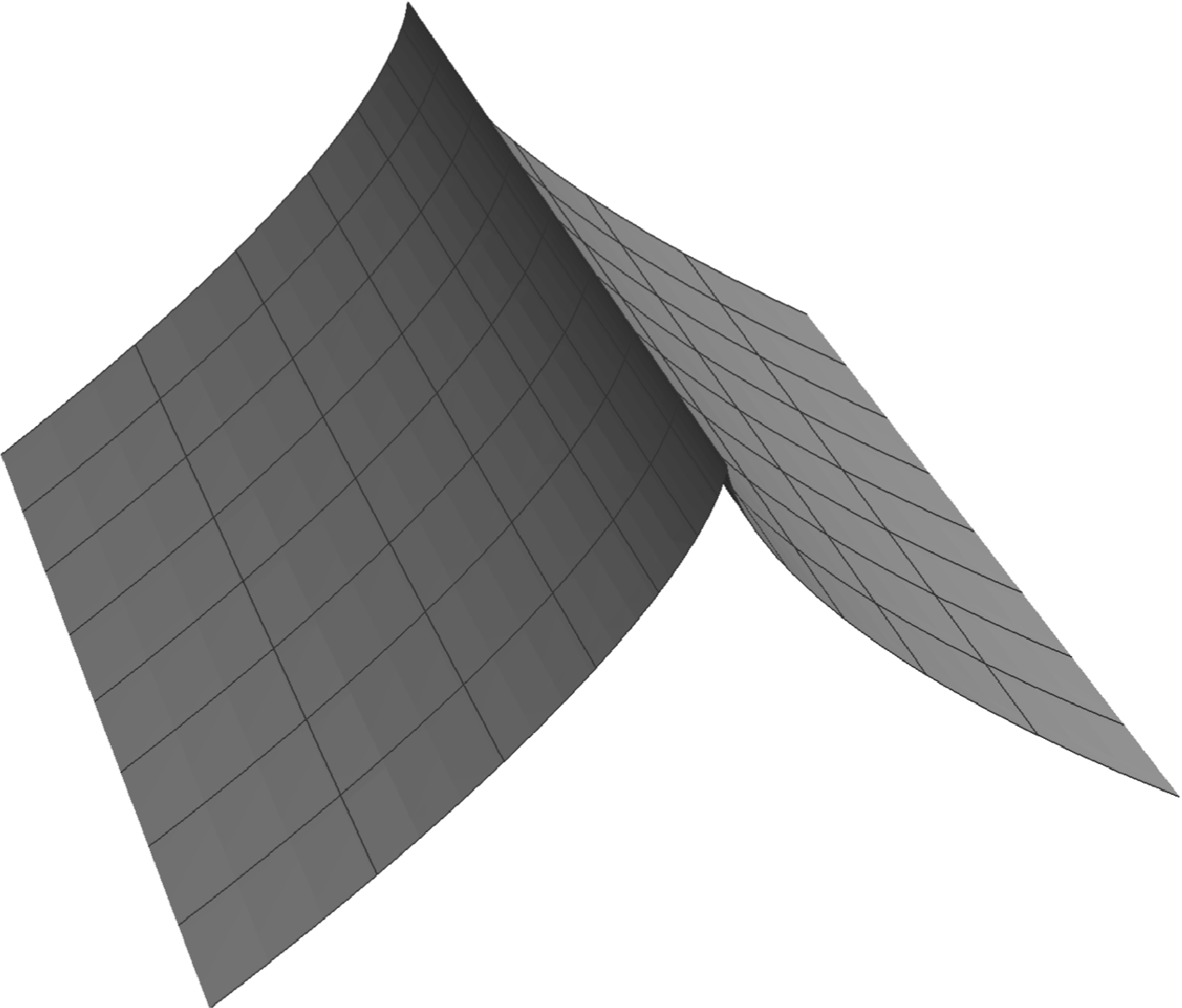}} 
 \end{tabular}
 \caption{The standard $5/2$-cuspidal edge (left) and cuspidal edge (right).}
 \label{fig:CEor}
\end{center}
\end{figure}

On the other hand,
one fundamental problem is
{to determine} the intrinsicity and extrinsicity of invariants. 
It is proved that {some basic invariants 
such as the singular curvature and the product curvature}
are intrinsic in \cite{SUY, MSUY}, and
they have various applications.
For example, 
the intrinsicity of the product curvature 
is used to prove existence
of isometric deformations of real analytic cuspidal edges
with non-vanishing limiting normal curvature
in \cite{NUY} and \cite{HNUY}.
{
See \cite{HHNSUY} for other applications.
In \cite{intcr, HNUY0},
several geometric invariants of cross caps
are proved to be intrinsic or extrinsic.}
In this paper, we determine
whether the above invariants of ${5/2}$-cuspidal edges
are intrinsic or extrinsic,
proving the existence
of isometric deformations of
real analytic ${5/2}$-cuspidal edges
with non-vanishing limiting normal curvature
as in \cite{NUY} and \cite{HNUY}.

This paper is organized as follows.
In Section \ref{sec:rhamph-cusps},
we define the secondary cuspidal curvature
and the bias for ${5/2}$-cusps,
and study their geometric properties.
In Section \ref{sec:rhamph-CE},
we deal with ${5/2}$-cuspidal edges
and define two invariants for them.
As an example, in Section \ref{sec:conjugate-Delaunay},
we calculate the invariants on the conjugate surfaces 
of spacelike Delaunay surfaces.
In Section \ref{sec:intrinsic},
we prove that the product 
(called the \emph{secondary product curvature\/}) of 
the secondary cuspidal curvature  
and the limiting normal curvature 
is an intrinsic invariant.
Using this intrinsicity, we show the existence
of isometric deformations of
real analytic ${5/2}$-cuspidal edges
with non-vanishing limiting normal curvature,
which yields the extrinsicity of various invariants,
see Table \ref{table:RCE}.
Finally, in Section \ref{sec:Kossowski},
we provide an intrinsic formulation of 
${5/2}$-cuspidal edges 
as a singular point of a positive semi-definite metric,
called the Kossowski metric.
Using an argument similar to that in Section \ref{sec:intrinsic},
we prove the existence of isometric realizations 
of Kossowski metrics with intrinsic ${5/2}$-cuspidal edges. 

\section{${5/2}$-cusps}
\label{sec:rhamph-cusps}
In this section, we discuss the geometric
properties of ${5/2}$-cusps.

\subsection{Invariants of ${5/2}$-cusps}
\label{sec:rhamphoidcusp}
Let $\gamma:(\R,0)\to(\R^2,0)$ be a map-germ,
and $\gamma'(0)=0$.
{We say that $\gamma$ is of {\it $A$-type\/} if
$\gamma''(0)\ne0$.}
Let $\gamma:(\R,0)\to(\R^2,0)$ be an $A$-type
map-germ.
The {\it cuspidal curvature\/} for $\gamma$ at $0$
is defined by
$$
\omega(\gamma,0)=
\dfrac{\det(\gamma''(0),\gamma'''(0))}
{|\gamma''(0)|^{5/2}},
$$
which measures a kind of wideness of $\gamma$ at $0$ (\cite{SUY2}).
{We may abbreviate $\omega(\gamma,0)$ 
as $\omega(\gamma)$, or $\omega$, in some cases.}
It is well known that an $A$-type map-germ $\gamma$
is a $3/2$-cusp if and only if 
$\det(\gamma''(0),\gamma'''(0))\neq0$,
and hence $\omega\ne0$.

Let $\gamma:(\R,0)\to(\R^2,0)$ be an $A$-type
map-germ with $\det(\gamma''(0),\gamma'''(0))=0$.
Then there exists $l\in \R$ such that
$$
\gamma'''(0)=l\gamma''(0).
$$
Then the {\it secondary cuspidal curvature\/} for 
$\gamma$ at $0$
is defined by
$$
\omega_r(\gamma,0)
=\dfrac{\det\Big(\gamma''(0),
3\gamma^{(5)}(0)-10l \gamma^{(4)}(0)\Big)}{|\gamma''(0)|^{7/2}}.
$$
We abbreviate $\omega_r=\omega_r(\gamma)=\omega_r(\gamma,0)$
as well.
By a direct calculation,
one can see that $\omega_r$ does not depend on the
parameter of $\gamma$.
The following criterion for ${5/2}$-cusp is known \cite{porteous}:
\begin{fact}\label{fact:cri25}
Let\/ $\gamma:(\R,0)\to(\R^2,0)$ be a map-germ with $\gamma'(0)=0$.
Then\/ $\gamma$ is a\/ ${5/2}$-cusp if and only if
\begin{enumerate}
\item\label{itm:25cond2} $\det(\gamma''(0),\gamma'''(0))=0$,
\item\label{itm:25cond3} 
$3\det(\gamma''(0),\gamma^{(5)}(0))\gamma''(0)
-10\det(\gamma''(0),\gamma^{(4)}(0))\gamma'''(0)\ne (0,0)$.
\end{enumerate}
\end{fact}
By the condition \ref{itm:25cond3},
$\gamma''(0)\ne0$.
When $\gamma$ is {of} $A$-type at $0$,
the conditions \ref{itm:25cond2} and \ref{itm:25cond3}
are written as follows.
By \ref{itm:25cond2}, there exists $l\in\R$ such that
$\gamma'''(0)=l \gamma''(0)$,
and then \ref{itm:25cond3} is written as
$$
\det\left(\gamma''(0),
3\gamma^{(5)}(0)-10l \gamma^{(4)}(0)
\right)\neq0.
$$
Thus an $A$-type germ $\gamma$
is a ${5/2}$-cusp if and only if
$\omega=0$ and $\omega_r\ne0$.

Next we define the bias of cusps.
Let $\gamma:(\R,0)\to(\R^2,0)$ be an $A$-type
map-germ which is not a $3/2$-cusp (i.e., $\omega=0$).
Then
$$
b(\gamma,0)
=\dfrac{\det(\gamma''(0),\gamma^{(4)}(0))}{|\gamma''(0)|^3}
$$
does not depend on the parameter,
and it is called the {\it bias} of cusps.
We abbreviate $b=b(\gamma)=b(\gamma,0)$
as well.
Let $\gamma$ be an $A$-type germ.
A line 
$$
\left\{u \lim_{t\to0}\dfrac{\gamma'(t)}{|\gamma'(t)|}
\,;\,u\in\R\right\}
=
\left\{u \gamma''(0)
\,;\,u\in\R\right\}
$$
passing through $\gamma(0)=0$
is called the {\it tangent line\/} of $\gamma$ at $0$.
We set two images of $\gamma$ as
$$
\gamma_+=\gamma((0,\ep)),\quad
\gamma_-=\gamma((-\ep,0)),
$$
for $\ep>0$.
We have the following proposition.

\begin{proposition}\label{prop:bias-G-meaning}
Let\/ $\gamma$ be an\/ $A$-type germ with $\omega=0$.
{If\/ $b\ne0$, 
then for a sufficiently small\/ $\ep>0$,}
the images\/ $\gamma_+$ and\/ $\gamma_-$ 
lie on the same side of the tangent line of\/ $\gamma$.
Moreover, if\/ $\gamma$ is a\/ ${5/2}$-cusp
and\/ {$b=0$, then 
for a sufficiently small\/ $\ep>0$,}
the images\/ $\gamma_+$ and\/ $\gamma_-$ 
lie on
both sides of the tangent line of\/ $\gamma$.
\end{proposition}

\begin{proof}
By rotating $\gamma$ and by a parameter change,
we may assume that
\begin{equation}\label{eq:cuspnormal}
\gamma=\left(\dfrac{t^2}{2},
\dfrac{t^4}{4!} \gamma_4+\dfrac{t^5}{5!}\gamma_5(t)\right),
\end{equation}
{where $\gamma_4\in\R$
and $\gamma_5(t)$ is a smooth function.}
Then $b=\gamma_4$ and
$\omega_r = 3\gamma_5(0)$.
{Since the tangent line is the horizontal axis,} 
the claim of the proposition is obvious by
these observations.
\end{proof}

One can easily see that
for $3/2$-cusps,
the images $\gamma_+$ and $\gamma_-$ 
always lie on both sides of the tangent line of $\gamma$.
Thus there is no similar notion of bias
for $3/2$-cusps.
If an $A$-type map-germ $\gamma$ with $\omega=0$
satisfies $b=0$, 
then $\gamma$ is said to be {\it balanced\/} 
(see Figure \ref{fig:rcusps}).

\begin{figure}[htb]
\begin{center}
 \begin{tabular}{{c@{\hspace{15mm}}c}}
  \resizebox{5cm}{!}{\includegraphics{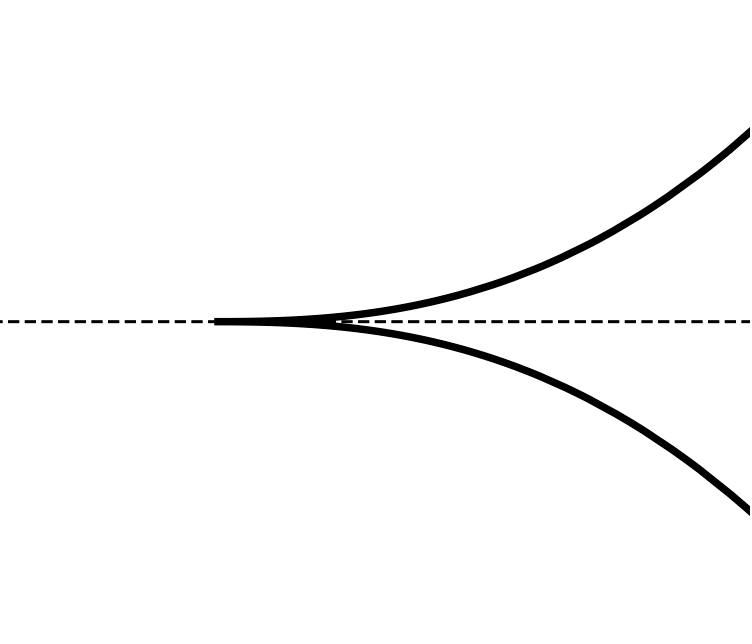}} &
  \resizebox{5cm}{!}{\includegraphics{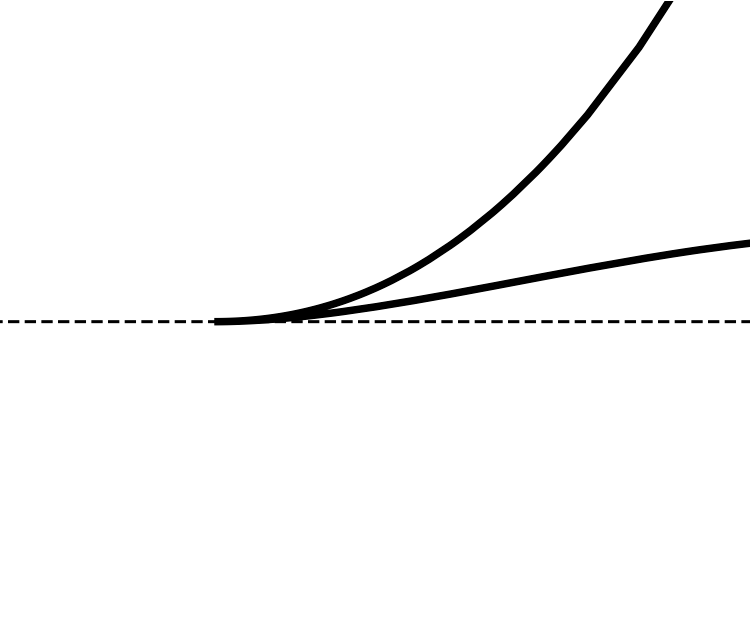}} 
 \end{tabular}
 \caption{The left figure shows a balanced $5/2$-cusp (i.e., $b=0$), 
       and the right one is non-balanced (i.e., $b\ne0$). 
       The dotted lines are the {tangent lines} at each singular point.
       As we have shown in Proposition \ref{prop:bias-G-meaning},
       the image of a balanced $5/2$-cusp extends over 
       the two domains separated by the tangent line.}
 \label{fig:rcusps}
\end{center}
\end{figure}

\subsection{Behavior of the curvature function}

Let $s_g$ be the {arclength} function 
$s_g(t)=\int_0^t|\gamma'(t)|\,dt$ of {an $A$-type}
germ $\gamma:(\R,0)\to(\R^2,0)$.
It is shown that
$(s(t):=)\, \sgn(t)\sqrt{|s_g(t)|}$
is $C^\infty$-differentiable and $s'(0)\neq0$
(\cite[Theorem 1.1]{su}).
Thus one can take $s(t)$
as a parameter, which is called {the {\it half-arclength parameter\/}}
\cite{su}.
We have the following proposition.
\begin{proposition}
Let\/ $\gamma:(\R,0)\to(\R^2,0)$ be a\/
${5/2}$-cusp, and\/ $t$ a parameter.
Let\/ $\kappa$ be the curvature defined everywhere except\/ $t=0$.
Then\/
$
\tilde\kappa=\sgn(t)\kappa
$
is a\/ $C^\infty$ function, and
$$
\tilde\kappa(0)=\dfrac{b}{3},\quad
\dfrac{d}{ds}\tilde\kappa(0)=
\dfrac{\sqrt{2}}{24}\omega_r
$$
holds, where\/ $s$ is the half-arclength parameter.
\end{proposition}

\begin{proof}
We may assume that $\gamma$ is given by
the form \eqref{eq:cuspnormal} without
loss of generality.
Then 
$$
\dfrac{\det(\gamma',\gamma'')}
{|\gamma'|^3}
=
\dfrac{
\dfrac{\gamma_4}{3}t^3+\dfrac{\gamma_5(0)}{8}t^4
+O(5)}
{\left|t^6+\dfrac{\gamma_4^2}{12}t^{10}+O(11)\right|^{1/2}}
=
\sgn(t)\left(\dfrac{\gamma_4}{3}+\dfrac{\gamma_5(0)}{8}t+O(2)\right)
$$
holds.
Here, $O(n)$ stands for the terms whose degrees are
greater than or equal to $n$.
On the other hand, by 
$|\gamma'|=\big|t\sqrt{1+t^4\gamma_4^2/36+O(5)}\big|=
|t+\gamma_4^2t^4/72+O(5)|$ and 
$s_g=t^2(1/2+O(4))$,
it holds that
$s=t\sqrt{1/2+O(4)}$ and
$dt/ds= \sqrt{2} $ at $0$.
The proposition is then obvious from the above calculations.
\end{proof}

See \cite{ft} for
another treatment of curvatures of curves with singularities.

\subsection{Projection of space curves}
Let $\Gamma:(\R,0)\to(\R^3,0)$ be a regular space curve,
and let $t$ be an {arclength parameter} of $\Gamma$, and
$\e, \n,\bb$ {the} Frenet frame.
We set {the orthogonal projection of $\Gamma$ 
to the normal plane $(\e(0))^\perp$ at $0$ by}
$$\gamma(t)=\Gamma(t)-\inner{\Gamma(t)}{\e(0)}\e(0).$$
Note that $\gamma'(0)=0$.
Then 
$\gamma$ at $0$ is $A$-type if and only if
$\kappa(0)\ne0$, where $\kappa$ is the curvature of $\Gamma$.
We assume that $\gamma$ is $A$-type
(i.e., $\kappa(0)\ne0$).
Since 
$$
\omega(\gamma,0)=\frac{\tau(0)}{\sqrt{\kappa(0)}},
$$ 
$\gamma$ at $0$ is a $3/2$-cusp if and only if
$\tau(0)\ne0$, where $\tau$ is the torsion of $\Gamma$.
If $\gamma$ is $A$-type but not a $3/2$-cusp
(i.e., $\kappa(0)\ne0$, $\tau(0)=0$), then
$$
b(\gamma,0)
=\frac{\tau'}{\kappa}(0),
\qquad
\omega_r(\gamma,0)=\frac{-\kappa'\tau'+3\kappa\tau''}{\kappa^{5/2}}(0).
$$
Thus, under the assumption $\kappa(0)\ne0$,
$\gamma$ is 
a $3/2$-cusp if and only if $\tau(0)\ne0$, and
$\gamma$ is 
not a $3/2$-cusp and non-balanced if and only if 
$\tau(0)=0$, $\tau'(0)\ne0$, and
$\gamma$ is 
a balanced $5/2$-cusp if and only if 
$\tau(0)=\tau'(0)=0$, $\tau''(0)\ne0$.

\section{Invariants of ${5/2}$-cuspidal edges}
\label{sec:rhamph-CE}
In this section, we discuss the geometric
properties of ${5/2}$-cuspidal edges.
\subsection{Frontals}
\label{sec:frontal}
Let $f:(\R^2,0)\to(\R^3,0)$ be a map-germ.
We call $f$ a {\it frontal\/} if there exists
a map $\nu:(\R^2,0)\to S^2$ satisfying
$\inner{df(X)}{\nu}=0$ for any $X\in T_p\R^2$ and
$p\in(\R^2,0)$,
where $S^2$ stands for the unit sphere in $\R^3$.
We call $\nu$ a unit normal vector field of $f$.
A frontal is called a {\it front\/} if $(f,\nu)$ is an immersion.
Let $f:(\R^2,0)\to(\R^3,0)$ be a frontal,
and $\nu$ a unit normal vector field of $f$.
We set 
\begin{equation}\label{eq:area-density}
\lambda=\det(f_u,f_v,\nu)
\end{equation}
by taking a coordinate system $(u,v)$,
with
$f_u=\partial f/\partial u$,
$f_v=\partial f/\partial v$.
We call $\lambda$ a {\it signed area density function}.
By the definition, 
$S(f)=\lambda^{-1}(0)$, where
$S(f)$ is the set of singularities of $f$.
A singular point $p$ of $f$ is
{said to be} {\it non-degenerate\/}
if $d\lambda(p)\ne0$.
If $p$ is a non-degenerate singular point,
then $S(f)$ near $p$ is a regular curve.
Let $p$ be a singular point satisfying $\rank df_p=1$,
then there exists a non-vanishing vector field $\eta$
on a neighborhood $U$ of $p$
such that $\langle \eta_q\rangle _{\R}=\ker df_q$
for $q\in S(f)\cap U$.
We call $\eta$ a {\it null vector field\/}.
We note that the notions of non-degeneracy and null vector field are 
introduced in \cite{KRSUY}.
We remark that a non-degenerate singular point satisfies
$\rank df_p=1$.
A non-degenerate singular point $p$ of $f$ is 
called
{\it first kind\/} (respectively, {\it second kind\/})
if $\eta_p$ is transverse to $S(f)$ at $p$
(respectively, $\eta_p$ is tangent to $S(f)$ at $p$).
It is well-known that a singular point of the first kind on a
front is a cuspidal edge 
(\cite[Proposition 1.3]{KRSUY}, see also \cite[Corollary 2.5]{SUYcamb}).
\subsection{Basic invariants for singular points of the first kind}
\label{sec:invs}
In \cite{SUY},
the singular curvature and the limiting normal curvature
are defined for cuspidal edges, namely
singular points of fronts of the first kind.
In \cite{MS,MSUY}, 
the cuspidal curvature and the cusp-directional torsion
are defined.
These definitions are also valid for singular points of
frontals of the first kind.

Let $f:(\R^2,0)\to(\R^3,0)$ be a frontal and $\nu$ a
unit normal vector field.
Let $0$ be a singular point of the first kind.
Taking a parametrization $\gamma:(\R,0)\to (\R^2,0)$
of $S(f)$,
the {\it singular curvature\/} $\kappa_s$ and
the {\it limiting normal curvature\/} $\kappa_\nu$
are defined by
$$
\kappa_s(t)=\sgn(d\lambda(\eta))
\dfrac{\det(\hat\gamma',\hat\gamma'',\nu\circ\gamma)}
{|\hat\gamma'|^3}(t),\quad
\kappa_\nu(t)=
\dfrac{\inner{\hat\gamma''}{\nu\circ\gamma}}
{|\hat\gamma'|^2}(t),
$$
respectively (\cite{SUY}), where $(\gamma',\eta)$ is
taken to be positively oriented.
Let $\xi$ be a vector field on $(\R^2,0)$ such that
$\xi_q$ is tangent to $S(f)$ at 
{each point} $q\in S(f)$,
and let $\eta$ be a null vector field.
Then the {\it cuspidal curvature\/} $\kappa_c$
and 
the {\it cusp-directional torsion\/} or
the {\it cuspidal torsion\/} $\kappa_t$ 
are defined by
\begin{align}
\label{eq:kappa-c}
\kappa_c(t)
&=
\frac{|\xi f|^{3/2} \det(\xi f,\ \eta^2 f,\
\eta^3 f)} {|\xi f\times\eta^2 f|^{5/2}}\Bigg|_{(u,v)=\gamma(t)},\\
\label{eq:kappa-t}
\kappa_t(t)
&=
\Biggl(
\frac{\det(\xi f,\,\eta^2 f,\,\xi\eta^2 f)}
{|\xi f\times\eta^2 f|^2} 
-
\frac{\det(\xi f,\,\eta^2 f,\,\xi^2 f)
\inner{\xi f}{\eta^2 f}}
{|\xi f|^2|\xi f\times\eta^2 f|^2}
\Biggr)
\Bigg|_{(u,v)=\gamma(t)},
\end{align}
where $\zeta^i f$ stands for the $i$'th order directional
derivative of $f$ by a vector field $\zeta$.
The invariant $\kappa_c$ measures a kind of ``wideness''
of the singularity. 
Furthermore, it is shown that $\kappa_\Pi=\kappa_c\kappa_\nu$
is an intrinsic invariant.
See Section \ref{sec:intrinsic} for the definition
of the intrinsicity and the extrinsicity of invariants.
See \cite{MSUY} for details.
One can easily see that
$\kappa_c(0)\ne0$ if and only if $f$ is a front at $0$.
It is known that for two cuspidal edges $f$ and $g$,
if their invariants
$\kappa_s,\kappa_s',\kappa_\nu,\kappa_\nu',\kappa_c,\kappa_t$ 
coincide at $0$,
then 
{there exists a coordinate system such that
3-jets $j^3f$ and $j^3g$ coincide at $0$} (\cite[Theorem 6.1]{MS}),
where ${}'=d/dt$ {and $t$ is an arclength parameter}.
In \cite{intcr, HHNSUY}, intrinsicities and extrinsicities 
of these invariants are investigated. 
See \cite{ot,fukui,ist} for another approach to investigating cuspidal edges,
and \cite{teramoto} for other applications of 
the above invariants (see also \cite{teramoto2}).

\subsection{Criterion and invariants for ${5/2}$-cuspidal edges}
{
First, we review the criterion for ${5/2}$-cuspidal edges
given in \cite[Theorem 4.1]{HKS}.
In order to do that,
we recall the following fact:

\begin{fact}[{\cite[Lemma 4.2]{HKS}}]
\label{fact:eta-tilde}
Let $f :(\R^2,0)\to (\R^3,0)$ be a frontal-germ
such that $0$ is a singular point of the first kind.
Let $\xi$, $\eta$ be vector fields on $(\R^2,0)$ 
such that the restriction $\xi|_{S(f)}$ is tangent to\/ $S(f)$,
and\/ $\eta$ is a null vector field.
Then there exists a null vector field $\tilde\eta$
such that 
\begin{equation}\label{eq:eta}
\inner{\xi f}{\tilde\eta^2 f}(0)=
\inner{\xi f}{\tilde\eta^3 f}(0)=0
\end{equation}
holds.
Moreover, if $\det(\xi f,\eta^2f,\eta^3 f)(0)=0$,
there exists the constant $l\in \R$
such that
\begin{equation}\label{eq:L}
\tilde\eta^3f(0)=l\, \tilde\eta^2f(0)
\end{equation}
holds.
\end{fact}

This fact is also obtained as a corollary of Lemma \ref{lem:eta-tilde}.}
Then the criterion for ${5/2}$-cuspidal edges
is given as follows:

\begin{proposition}[Criterion for ${5/2}$-cuspidal edges, 
{\cite[Theorem 4.1]{HKS}}]
\label{prop:criteria}
The frontal-germ\/ $f :(\R^2,0)\to (\R^3,0)$ is a\/
${5/2}$-cuspidal edge if and only if
\begin{enumerate}
\item\label{itm:etal} $\eta\lambda(0)\ne0$,
\item\label{itm:no23} 
$\det(\xi f,\eta^2f,\eta^3 f)=0$ on\/ $S(f)$,
\item\label{itm:25new}
$\det
(\xi f,\tilde\eta^2 f,3\tilde\eta^5f-10l\tilde\eta^4 f)
(0)\ne0$.
\end{enumerate}
Here, $\xi$ is a vector field on $(\R^2,0)$
such that the restriction $\xi|_{S(f)}$ is tangent to\/ $S(f)$,
and\/ $\eta$ is a null vector field.
Furthermore,
{$\tilde\eta$ is a null vector field 
and $l\in \R$ is the constant
given in Fact $\ref{fact:eta-tilde}$}.
\end{proposition}

The condition \ref{itm:etal}
implies that $0$ is a singular point of the first kind.
{
Moreover, by \cite[Proposition 3.11]{MSUY},
the condition \ref{itm:no23} implies 
that $f$ is not a front:

\begin{fact}[{\cite[Proposition 3.11]{MSUY}}]
\label{fact:front-condition}
Let $f :(\R^2,0)\to (\R^3,0)$ be a frontal-germ
such that $0$ is a singular point of the first kind.
Take vector fields $\xi$, $\eta$ on $(\R^2,0)$ 
such that the restriction $\xi|_{S(f)}$ is tangent to\/ $S(f)$,
and\/ $\eta$ is a null vector field.
Let $q\in S(f)$ be a singular point of the first kind.
Then,
$$\det(\xi f,\eta^2f,\eta^3 f)(q)\ne0$$ holds
if and only if 
$f$ is a front at $q\in S(f)$.
\end{fact}
}
\begin{proof}
{A frontal-germ $f$ at a non-degenerate singular point $q\in S(f)$ 
is a front if and only if $\eta\nu(q)=0$,
where $\eta$ is a null vector field and $\nu$ is
a unit normal vector field.}
Firstly we show that the condition
\ref{itm:no23} is equivalent to
$\det(\xi f,\nu,\eta\nu)=0$.
Since 
$\inner{\nu}{\xi f}=0$ and
$\inner{\nu}{\eta^2 f}=-\inner{\eta \nu}{\eta f}=0$,
we see that $\nu$ is parallel to $\xi f\times \eta^2 f$
on $S(f)$.
Thus
\begin{equation}\label{eq:onaji1}
\begin{array}{rcl}
|\xi f\times \eta\eta f|^2\det(\xi f,\nu,\eta\nu)
&=&
\det(\xi f,\xi f\times \eta^2 f,\eta(\xi f\times \eta^2 f))\\
&=&
\det(\xi f,\xi f\times \eta^2 f,
\eta\xi f\times \eta^2 f+\xi f\times \eta^3 f).
\end{array}
\end{equation}
{Since $[\eta,\xi]$ is a vector field and $df(T_p)$ is
generated by $\xi_p$ $(p\in S(f))$,
the derivative} $[\eta,\xi]f=\eta\xi f-\xi\eta f$ is parallel to
$\xi f$, and $\eta f=0$ on $S(f)$.
Since $\xi$ is tangent to $S(f)$, $\xi\eta f=0$ 
on $S(f)$.
Hence $\eta\xi f$ is parallel to $\xi f$.
Thus the left-hand side of \eqref{eq:onaji1} is
equal to
\begin{equation}\label{eq:onaji2}
\det(\xi f,\xi f\times \eta^2 f,
\xi f\times \eta^3 f)(t).
\end{equation}
Since
$\det(a,a\times b,a\times c)=|a|^2\det(a,b,c)$
for vectors $a,b,c\in\R^3$,
\eqref{eq:onaji2} is a non-zero multiple of
$
\det(\xi f,\eta^2 f,\eta^3 f)(t).
$
Thus
\ref{itm:no23} is equivalent to
$\det(\xi f,\nu,\eta\nu)=0$.
One can write $\eta\nu=\alpha \xi f+\beta\nu$.
Then $\beta=\inner{\eta\nu}{\nu}=0$.
On the other hand, 
$|\xi f|^2\alpha=\inner{\eta\nu}{\xi f}$.
Since $\inner{\nu}{\xi f}(u,v)\inner{\nu}{\eta f}(u,v)=0$ for any $(u,v)$,
it holds that 
$\inner{\eta\nu}{\xi f}(u,v)+\inner{\nu}{\eta\xi f}(u,v)=0$ 
$\inner{\xi\nu}{\eta f}(u,v)+\inner{\nu}{\xi\eta f}(u,v)=0$ 
for any $(u,v)$.
Since $[\eta,\xi]$ is a vector field, $[\eta,\xi]f$ is parallel to
$\xi f$ at $0$.
Thus $\inner{\nu}{[\eta,\xi]f}(0)=0$.
Hence we have
$|\xi f|^2\alpha
=\inner{\eta\nu}{\xi f}=-\inner{\nu}{\eta\xi f}
=-\inner{\nu}{\xi\eta f}=\inner{\xi\nu}{\eta f}=0$.
This completes the proof.
\end{proof}

{
To define the invariants of $5/2$-cuspidal edges,
we prepare the following lemma:

\begin{lemma}
\label{lem:eta-tilde}
Let $f :(\R^2,0)\to (\R^3,0)$ be a frontal-germ
such that $0$ is a singular point of the first kind.
Assume that each singular point $q\in S(f)$ is of the first kind.
Let $\gamma(t)$ be a parametrization of $S(f)$
such that $\gamma(0)=0$,
and let $\xi$ be a vector field on $(\R^2,0)$ 
such that the restriction $\xi|_{S(f)}$ is tangent to\/ $S(f)$.
Then, there exists a null vector field $\tilde\eta$
such that 
\begin{equation}\label{eq:eta_t}
\inner{\xi f}{\tilde\eta^2 f}(\gamma(t))=
\inner{\xi f}{\tilde\eta^3 f}(\gamma(t))=0
\end{equation}
holds along $\gamma(t)$.
Moreover, 
if $f$ is not a front at each $q\in S(f)$,
then there exists a function $l(t)$
such that
\begin{equation}\label{eq:L_t}
\tilde\eta^3f(\gamma(t))=l(t)\, \tilde\eta^2f(\gamma(t))
\end{equation}
holds along $\gamma(t)$.
\end{lemma}

\begin{proof}
We take a coordinate system $(u,v)$ satisfying
$S(f)=\{v=0\}$, $\eta=\partial_v$.
Set
\begin{equation}\label{eq:eta-alpha}
  \tilde{\eta}
  = \alpha(u,v)\partial_u+\partial_v\qquad
  (\alpha(u,v)=v(\alpha_1(u)+\alpha_2(u)v)),  
\end{equation}
where we set
$$
  \alpha_1(u) 
  = -\left.\frac{\inner{f_{vv}}{f_u}}{\inner{f_u}{f_u}}\right|_{v=0},
  \quad
  \alpha_2(u) 
  = -\left.\frac{3\alpha_1(u)\inner{f_{uv}}{f_u}
  +\inner{f_{vvv}}{f_u}}{2\inner{f_u}{f_u}}\right|_{v=0}.
$$
We can check that 
$\inner{f_u}{\tilde\eta^2f}(u,0)
=\inner{f_u}{\tilde\eta^3f}(u,0)=0$.
With respect to the second assertion,
we have 
$\det(\xi f,\tilde\eta^2f,\tilde\eta^3 f)=0$ along $S(f)$,
by Fact \ref{fact:front-condition}.
Hence, there exist functions $l(t)$, $\bar{l}(t)$
such that  
$$\tilde\eta^3 f(\gamma(t))
= l(t)\, \tilde\eta^2f(\gamma(t))
+\bar{l}(t)\, \xi f(\gamma(t)).$$
Since $\inner{\xi f}{\tilde\eta^3 f}(\gamma(t))=0$,
we have $\bar{l}(t)=0$.
\end{proof}

\begin{corollary}\label{cor:tilde-coord}
Let $f :(\R^2,0)\to (\R^3,0)$ be a frontal-germ
such that $0$ is a singular point of the first kind.
Assume that each singular point $q\in S(f)$ is of the first kind.
Then, there exists a coordinate system $(u,v)$
such that $S(f)=\{v=0\}$, and
$$
  f_v=0, \qquad
  \inner{f_u}{f_{vv}}
  =\inner{f_u}{f_{vvv}}=0
$$
along the $u$-axis.
\end{corollary}

\begin{proof}
We take a coordinate system $(u,v)$ satisfying
$S(f)=\{v=0\}$, $\eta=\partial_v$.
Set $\tilde{\eta}$ as in \eqref{eq:eta-alpha}.
Then there exists a coordinate system 
$(x,y)$
such that $x=u$
and $\partial_{y}$ is parallel to $\tilde{\eta}$.
Since 
$\inner{f_x}{\tilde\eta^2f}(x,0)
=\inner{f_x}{\tilde\eta^3f}(x,0)=0$,
we have 
$\inner{f_x}{f_{yy}}(x,0)
=\inner{f_x}{f_{yyy}}(x,0)=0$.
Hence $(x,y)$ is the desired coordinate system.
\end{proof}
}

Let us assume $S(f)$ is oriented, and let
$\xi$ be a vector field such that the restriction $\xi|_{S(f)}$
is tangent to $S(f)$ agreeing with the orientation of $S(f)$, and
let $\eta$ be a null vector field so that $(\xi,\eta)$ is
positively oriented.
We take a null vector field $\tilde\eta$ which
satisfies the condition \eqref{eq:eta}, and $(\xi,\tilde\eta)$ is
positively oriented.
Assuming $f$ is not a front at $0$, 
then {by Fact \ref{fact:eta-tilde},
there exists a number $l\in \R$ such that} 
$\tilde\eta^3f(0,0)=l \tilde\eta^2f(0,0)$.
We define two real numbers at $0$, respectively, by
\begin{align*}
r_b&=
\dfrac{|\xi f(0,0)|^2\det\Big(\xi f(0,0),\ 
\tilde\eta^2 f(0,0),\ \tilde\eta^4f(0,0)\Big)}
{|\xi f(0,0)\times \tilde\eta^2 f(0,0)|^{3}},\\
r_c&=
\dfrac{|\xi f(0,0)|^{5/2}\det\Big(
\xi f(0,0),\ \tilde\eta^2 f(0,0),\ 
3\tilde\eta^5 f(0,0)-10\,l\,\tilde\eta^4 f(0,0)
\Big)}
{|\xi f(0,0)\times \tilde\eta^2 f(0,0)|^{7/2}}.
\end{align*}
\begin{lemma}\label{lem:welldef}
The two real numbers\/ $r_b$, and\/ $r_c$ do not depend on 
the choices of\/ $\xi$
and\/ $\tilde\eta$.
\end{lemma}

\begin{proof}
We take the coordinate system $(u,v)$ 
{given in Corollary \ref{cor:tilde-coord}.}
We set
$$
\xi=\aaa(u,v)\partial_u+\ab(u,v)\partial_v,\quad
\overline\eta=\ac(u,v)\partial_u+\ad(u,v)\partial_v,
$$
where $\alpha_i(u,v)$ $(i=1,2,3,4)$ 
is a smooth function such that $\aaa,\ad>0, \ac(u,0)=0$.
By the assumption \eqref{eq:eta},
$(\ac)_v=(\ac)_{vv}=0$ holds on the $u$-axis.
By a straightforward calculation,
$$
\begin{array}{rcl}
\xi f&=&\aaa f_u,\\
\overline\eta^2f&=&\ad^2 f_{vv}\\
\overline\eta^3f
&=&
\ad^2 (3 (\ad)_v f_{vv}+\ad f_{vvv}),\\
\overline\eta^4f
&=&
\ad^4f_{vvvv}+f_u* +f_{vv}*
\end{array}
$$
hold on the $u$-axis.
Thus 
$$
\dfrac{|f_{\xi}|^2\det\Big(
\xi f,\ \overline\eta^2f,\ 
\overline\eta^4 f\Big)}
{|\xi f\times \overline\eta^2 f|^3}
=
\dfrac{\det\Big(
\aaa f_u,\ \ad^2 f_{vv},\ \ad^4f_{vvvv}\Big)}
{|\aaa f_u||\ad^2 f_{vv}|^3}
=
\dfrac{\det(
f_u,\ f_{vv},\ f_{vvvv})}
{|f_u||f_{vv}|^3}
$$
holds at $0$, which shows the independence of $r_b$.
By the above calculation, 
if $f_{vvv}=lf_{vv}$, then
$\overline\eta^3f=(3(\ad)_v+\ad l)\overline\eta^2f$.
Moreover, we see
$$
\overline\eta^5f=
\ad^4 (10 (\ad)_v f_{vvvv}+\ad f_{vvvvv}).
$$
Hence, 
{\allowdisplaybreaks
\begin{align*}
&\dfrac{|\xi f|^{5/2}\det\Big(
\xi f,\ \overline\eta^2f,\ 
3\overline\eta^5f-10 l \overline\eta^4f\Big)}
{|\xi f\times \overline\eta^2f|^{7/2}}\\[5mm]
=&
\dfrac{\det\Big(
\aaa f_u,\ \ad^2 f_{vv},\ 
3\ad^4 (10 (\ad)_v f_{vvvv}+\ad f_{vvvvv})-10 (3(\ad)_v+\ad l) \ad^4f_{vvvv}\Big)}
{|\aaa f_u||\ad^2 f_{vv}|^{7/2}}\\[5mm]
=&
\dfrac{\det\Big(
f_u,\  f_{vv},\ 
3 f_{vvvvv}-10 l f_{vvvv}\Big)}
{|f_u||f_{vv}|^{7/2}}
\end{align*}
}%
holds at $0$, which shows the independence of $r_c$.
\end{proof}

\begin{definition}
{Let $f :(\R^2,0)\to (\R^3,0)$ be a frontal-germ
such that $0$ is a singular point of the first kind.}
Assume that each singular point is of the first kind 
and is not a front. 
For example, ${5/2}$-cuspidal edges satisfy this assumption.
Then, {by Lemma \ref{lem:eta-tilde}, we have}
$\tilde\eta^3f(\gamma(t))=l(t) \tilde\eta^2f(\gamma(t))$,
where $\gamma(t)$ is a parametrization of $S(f)$,
and $\tilde\eta$ is a null vector field satisfying
$\inner{\xi f}{\tilde\eta^2 f}(\gamma(t))=
\inner{\xi f}{\tilde\eta^3 f}(\gamma(t))=0$.
Then we define $r_b(t)$, $r_c(t)$ as
\begin{align}
r_b(t)&=
\dfrac{|\xi f|^2\det\Big(\xi f,\ 
\tilde \eta^2f,\ \tilde\eta^4f\Big)}
{|\xi f\times \tilde\eta^2f|^{3}}\Bigg|_{(u,v)=\gamma(t)},\\
\label{eq:kappa-c-ni}
r_c(t)&=
\dfrac{|\xi f|^{5/2}\det\Big(
\xi f,\ \tilde\eta^2f,\ 
3\tilde\eta^5f-10\,l\,\tilde\eta^4f
\Big)}
{|\xi f\times \tilde\eta^2f|^{7/2}}\Bigg|_{(u,v)=\gamma(t)},
\end{align}
respectively.
The invariant $r_b(t)$ is called the {\it bias},
and $r_c(t)$ is called the {\it secondary cuspidal curvature}.
We also define
$$
r_\Pi(p):=\kappa_{\nu}(p)r_c(p)
$$
for a singular point $p$,
which is called the {\it secondary product curvature}.
\end{definition}

In \cite{OSaji}, 
the bias $r_b$ and the secondary cuspidal curvature $r_c$ 
are used to investigate 
the cuspidal cross caps.

\subsection{Geometric meanings}

Here we study geometric meanings of
the invariants $r_b$ and $r_c$.

{
Let $f:(\R^2,0)\to(\R^3,0)$ be a frontal 
and $0$ a singular point of the first kind.
Moreover, let $\gamma(t)$ $(\gamma(0)=0)$ 
be a parametrization of $S(f)$,
and we set $\hat{\gamma}(t):=f(\gamma(t))$.
Since $0$ is a singular point of the first kind,
$\hat{\gamma}'(0)\ne0$, 
{where ${}'=d/dt$}.
Denote by $\Pi_f$ 
the normal plane $( \hat{\gamma}'(0))^\perp$ 
of $ \hat{\gamma}'(0)$ passing through $0$.
We call $\Pi_f$ 
the {\it normal plane} of $f$ passing through $0$.

\begin{proposition}
Let $f:(\R^2,0)\to(\R^3,0)$ be a frontal 
and $0$ a singular point of the first kind.
Assume that $f$ is not a front.
Let $r_b$ 
{\rm (}respectively, $r_c${\rm )}
be the bias 
{\rm (}respectively, the secondary cuspidal curvature{\rm )} 
of the frontal $f$ at $0$.
Denote by $\Pi_f$ 
the normal plane of $f$ passing through $0$.
Then, 
\begin{itemize}
\item
the slice of $f$ by the normal plane $\Pi_f$
is an image of an A-type map-germ $\hat c:(\R,0) \to (\Pi_f,0)$.
Moreover,
\item
if we denote by 
$b(\hat c,0)$
{\rm (}respectively, $\omega_r(\hat c,0)${\rm )} 
the bias of cusps
{\rm (}respectively, the secondary cuspidal curvature{\rm )} 
of $\hat c$ at $0$ as a plane curve in $\Pi_f$,
then we have
$$
  r_b = b(\hat c,0),\qquad
  r_c = \omega_r(\hat c,0).
$$
\end{itemize}
\end{proposition}

\begin{proof}
Let $\gamma(t)$ $(\gamma(0)=0)$ 
be a parametrization of $S(f)$,
and we set $\hat{\gamma}(t):=f(\gamma(t))$.
The slice of $f$ by the normal plane $\Pi_f=( \hat{\gamma}'(0))^\perp$
is given by
$$
C=\{(u,v)\,;\,\inner{f(u,v)}{ \hat{\gamma}'(0)}=0\},
$$
where ${}'=d/dt$.
We take a coordinate system satisfying 
$S(f)=\{v=0\}$, $\eta=\partial_v$ and
$\inner{f_u}{f_{vv}}=\inner{f_u}{f_{vvv}}=0$ on $S(f)$
(Corollary \ref{cor:tilde-coord}).
Then
we see that $\inner{f(u,v)}{ \hat{\gamma}'(0)}_u\ne0$ at $0$.
Thus we can take a parametrization of $C$
as $c(v)=(c_1(v),v)$. We set $\hat c=f\circ c$.
We remark that since $\inner{\hat c(v)}{\hat\gamma'(0)}=0$,
it holds that $c_1'(0)=0$. Furthermore, since $f_v(u,0)=0$,
it holds that $f_{uv}(u,0)=f_{uuv}(u,0)=f_{uuuv}(u,0)=0$.
Then we have
$$
\hat c''(0)=f_{vv}(0,0)+c_1''(0)f_u(0,0),\quad
\hat c'''(0)=f_{vvv}(0,0)+c_1'''(0)f_u(0,0).
$$
Since $\inner{\hat c(v)}{\hat\gamma'(0)}=0$,
it holds that 
$c_1''(0)=-\inner{f_u(0,0)}{f_{vv}(0,0)}=0$
and
$c_1'''(0)=-\inner{f_u(0,0)}{f_{vvv}(0,0)}=0$.
Furthermore, since
$$
\hat c^{(4)}(0)=f_{vvvv}(0,0)+c_1^{(4)}(0)f_u(0,0),\quad
\hat c^{(5)}(0)=f_{vvvvv}(0,0)+c_1^{(5)}(0)f_u(0,0),
$$
we see that
\begin{align*}
b(\hat c,0)&=\dfrac{\det(f_u,f_{vv},f_{vvvv})}
{|f_{vv}|^3}(0,0)=r_b,\\
\omega_r(\hat c,0)&=
\dfrac{\det(f_u,f_{vv},3f_{vvvvv}-10lf_{vvvv})}
{|f_{vv}|^{7/2}}(0,0)=r_c.
\end{align*}
\end{proof}
}

\subsection{Normal form for ${5/2}$-cuspidal edges}
In \cite{MS}, a normal form for cuspidal edges
is given. 
We have the following.
\begin{proposition}\label{prop:normalform}
Let\/ $f:(\R^2,0)\to(\R^3,0)$ be a\/ ${5/2}$-cuspidal edge.
Then there exist a coordinate system\/ $(u,v)$ on\/ $(\R^2,0)$
and an isometry\/ $\Phi:(\R^3,0)\to (\R^3,0)$ such that
\begin{multline}\label{eq:normalform}
\displaystyle
\Phi\circ f(u,v)
=
\Bigg(
u,\ \sum_{i=2}^5 \dfrac{a_i}{i!}u^i+\dfrac{v^2}{2},
\displaystyle
\sum_{i=2}^5 \dfrac{b_{i0}}{i!}u^i\\
 + 
\sum_{i=1}^3 \dfrac{b_{i2}}{i!}u^iv^2
+
\dfrac{b_{14}}{4!}uv^4
+
\sum_{i=4}^5 \dfrac{b_{0i}}{i!}v^i
\Bigg)
+h(u,v),
\end{multline}
where\/ 
{$a_i\in\R$ $(i=2,\ldots,5)$,
$b_{ij}\in\R$ $(i+j\leq5)$
are constants satisfying\/ $b_{05}\ne0$, and}\/
$h(u,v)$ consists of the terms
whose degrees are greater than or equal to\/ $6$, 
of the form
$$\big(0,\ u^6h_1(u),\ u^6h_2(u)+u^4v^2h_3(u)
+u^2v^4h_4(u)
+ uv^5h_5(u)+v^6h_6(u,v)\big).$$
\end{proposition}
Although this proposition can be shown by the same method
as in the proof of \cite[Theorem 3.1]{MS},
we give a proof in Appendix \ref{sec:coord} for the sake of
completeness.
Under this normal form, the
invariants defined above can be computed
as
\begin{itemize}
\item
$(\kappa_\nu(0),\kappa_\nu'(0),\kappa_\nu''(0),\kappa_\nu'''(0))
=
\Big(\mbox{\fbox{$b_{20}$}},\,
\mbox{\fbox{$b_{30}$}} -2 a_2 b_{12},\,
\mbox{\fbox{$b_{40}$}}
-4 a_3 b_{12}
-2 a_2^2 b_{20}
-2 a_2 b_{22}
-3 b_{20}^3
-4 b_{12}^2 b_{20},\,
\mbox{\fbox{$b_{50}$}}
+14 a_2^3 b_{12}
-7 a_2^2 b_{30}
-6 a_4 b_{12}-6 a_3 b_{22}-12 b_{12} b_{20} b_{22}
-12 b_{12}^2 b_{30}-19 b_{20}^2 b_{30}
+a_2 (-6 a_3 b_{20}-2 b_{32}+24 b_{12}^3+32 b_{20}^2 b_{12})
\Big)$,
\item
$(\kappa_s(0),\kappa_s'(0),\kappa_s''(0),\kappa_s'''(0))
=
\Big(
\mbox{\fbox{$a_2$}},\,
\mbox{\fbox{$a_3$}}
+2 b_{12} b_{20},\,
\mbox{\fbox{$a_4$}}
-4 a_2 (b_{12}^2+b_{20}^2)
+2 b_{20} b_{22}
+4 b_{12} b_{30}
-3 a_2^3,\,
\mbox{\fbox{$a_5$}}
-a_2^2 (8 b_{12} b_{20}+19 a_3)
-2 a_3 (6 b_{12}^2+5 b_{20}^2)
-3 a_2 (4 b_{12} b_{22}+5 b_{20} b_{30})
-24 b_{20} b_{12}^3
+2 b_{12}(3 b_{40}-13 b_{20}^3) 
+6 b_{22} b_{30}+2 b_{20} b_{32}
\Big)$,
\item
$(\kappa_t(0),\kappa_t'(0),\kappa_t''(0))
=
\Big(
\mbox{\fbox{2$b_{12}$}},\,
\mbox{\fbox{$2 b_{22}$}}
-a_2 b_{20},\,
\mbox{\fbox{$2b_{32}$}}
+4 a_2^2 b_{12}
-a_3 b_{20}
-2 a_2 b_{30}
-16 b_{12}^3
-8 b_{20}^2 b_{12}
\Big)$,
\item
$(r_b(0),\, r_b'(0))
=
\Big(
\mbox{\fbox{$b_{04}$}},\,
\mbox{\fbox{$b_{14}$}}-12 a_2 b_{12}
\Big)$,
\item
$r_c(0)
=\mbox{\fbox{$3b_{05}$}}$,
\end{itemize}
and $\kappa_c\equiv0$,
where the prime
means differentiation 
with respect to the {arclength} parameter of $\hat\gamma$.
Looking at the boxed entries, we have the following proposition.
\begin{proposition}
Let\/ $f,g$ be germs of ${5/2}$-cuspidal edges.
If their invariants\/
$\kappa_\nu$,
$\kappa_\nu'$,
$\kappa_\nu''$,
$\kappa_\nu'''$,
$\kappa_s$, $\kappa_s'$, $\kappa_s''$, $\kappa_s'''$,
$\kappa_t$, $\kappa_t'$, $\kappa_t''$,
$r_b$, $r_b'$, $r_c$ 
coincide at\/ $0$,
then there exist a coordinate system\/ $(u,v)$
and an isometry\/ $A$ of\/ $\R^3$ 
such that 
$$
j^5_0f(u,v)=j^5_0(A\circ g)(u,v),
$$
where\/ $j^5_0f(u,v)$ stands for
the\/ $5$-jet of\/ $f$ with respect to\/ $(u,v)$ at\/ $0$.
\end{proposition}
Moreover, a parametrization of 
$f(S(f))$ as a space curve is given by $f(u,0)$.
Since $b_{04},b_{14},b_{05}$ do not appear in
$f(u,0)$, they also do not appear in
the curvature $\kappa$ and the torsion $\tau$ of $f(u,0)$.
Thus we believe that the invariants $r_b,r_c$ for 
${5/2}$-cuspidal edges were not 
{paid attention to} before.

\subsection{Invariants of ${5/2}$-cuspidal edges on conjugate surfaces}
\label{sec:conjugate-Delaunay}
We denote by $\R^3_1$ the Lorentz-Minkowski $3$-space 
with signature $(-,+,+)$.
A {\it spacelike Delaunay surface with axis $\ell$}
is a surface in $\R^3_1$ such that
the first fundamental form (that is, the induced metric) is positive definite,
it is of constant mean curvature ({\it CMC}, for short), and
it is invariant under the action of the group of motions in $\R^3_1$
which fixes each point of the line $\ell$.
Such spacelike Delaunay surfaces 
are classified and they have 
{\it conelike singularities} (see \cite{Honda}, for details).

As in the case of CMC surfaces in $\R^3$,
for a given (simply-connected) spacelike CMC surface in $\R^3_1$,
there exists a spacelike CMC surface called the {\it conjugate}.
Any conjugate surface of a spacelike Delaunay surface
is a spacelike {\it helicoidal} CMC surface\footnote{%
A {\it helicoidal surface} is a surface 
which is invariant under 
a non-trivial one-parameter subgroup 
of the isometry group of $\R^3_1$.},
and it is shown in \cite{HKS} 
that such spacelike helicoidal CMC surfaces
have ${5/2}$-cuspidal edges.
We remark that 
spacelike zero-mean-curvature surfaces
(i.e., {\it maximal} surface)
never admit ${5/2}$-cuspidal edges 
(cf.\ \cite{HKS}, see also \cite{UY_hokudai}).

In this section, we compute 
the invariants $r_b$ and $r_c$ of ${5/2}$-cuspidal edges
on such spacelike helicoidal CMC surfaces, 
regarding them as surfaces in $\R^3$.
More precisely, 
setting 
\begin{equation}\label{eq:delta}
  \delta(u) = \left(u^2+k+1\right)^2-4 k,
\end{equation}
a non-totally-umbilical spacelike Delaunay surface with timelike axis
is given by
$$
  f_{\rm Del}(u, v) = \frac1{2H}\left(
    \int_{0}^u \frac{\tau^2+k-1}{\sqrt{\delta(\tau)}} d\tau,\, 
    u \cos (2Hv),\, u \sin (2Hv) 
  \right)
$$
for some constant $k\in \R$ $(k\ne1)$,
where $H$ is the mean curvature
(see \cite{HKS} for more details).
Let $f : (\R^2,0) \to (\R^3_1,0)$ 
be a spacelike helicoidal CMC surface
which is given as a conjugate surface 
of the Delaunay surface $f_{\rm Del}$.
Setting $\Delta(u):=\delta(u)-u^4$, 
such an $f$ can be written as follows (cf.\ \cite{HKS}):
\begin{enumerate}
\item
If $-1<k<1$ or $1<k$, then $f$ is congruent to 
\begin{equation}\label{eq:t-conj}
  f_{{T}}(u, v) 
     = \left( \psi +\frac{1-k}{2H (1+k)}\, \phi,\, \rho \cos \phi,\, \rho \sin \phi \right),
\end{equation}
where 
{\allowdisplaybreaks
\begin{align*}
  &\rho(u) := \frac{\sqrt{\Delta (u)}}{2 H (k+1)},\quad
   \psi (u) := 
     \int_{0}^u \frac{\sqrt{ 2(1+k) }\, \tau^4}{
                  H \sqrt{\delta (\tau)} \Delta (\tau)}d\tau,\\
  &\phi(u,v) := 
      \int_{0}^u  \frac{\sqrt{ 2(1+k) } (1-k) \tau^2}{
      \sqrt{\delta (\tau)} \Delta (\tau)}  d\tau
       - \sqrt{\frac{1+k}{2}}v.
\end{align*}}
\item
If $k<-1$, then $f$ is congruent to 
\begin{equation}\label{eq:s-conj}
    f_{{S}}(u, v) = 
    \left(\rho \sinh \phi,\, \rho \cosh \phi,\,  \psi +\frac{k-1}{2H (1+k)}\, \phi\right),
\end{equation}
where 
{\allowdisplaybreaks
\begin{align*}
  &\rho(u) 
     := -\frac{\sqrt{\Delta (u)}}{2 H (1+k)} ,\quad
   \psi (r) := 
     \int_{0}^u \frac{\sqrt{ 2(-k-1) }\, \tau^4}{
     H \sqrt{\delta (\tau)} \Delta (\tau)}d\tau,\\
  &\phi(u,v) := 
      -\int_{0}^u  \frac{\sqrt{ 2(-k-1) } (1-k) \tau^2}{
      \sqrt{\delta (\tau)} \Delta (\tau)}  d\tau
     - \sqrt{\frac{-k-1}{2}}v.
\end{align*}}
\item
If $k=-1$, then $f$ is congruent to 
\begin{equation}\label{eq:l-conj}
    f_{{L}}(u, v) = ( \psi -\rho-\rho \phi^2,\, -2\rho\phi,\,  \psi +\rho-\rho \phi^2)
  +H \left(\frac{\phi^3}{3}+\phi,\, \phi^2,\, \frac{\phi^3}{3}-\phi\right),
\end{equation}
where $\rho(u) := u/2$,
\begin{equation*}
   \psi (u) 
   := \int_{0}^u \frac{\tau^2(\sqrt{\tau^4+4}+\tau^2)}{4 H^2 \sqrt{\tau^4+4}} d\tau
   ,\quad
  \phi(u,v) 
  := \int_{0}^u \frac{\sqrt{\tau^4+4}+\tau^2}{2H \sqrt{\tau^4+4}} d\tau + v.
\end{equation*}
\end{enumerate}

\begin{figure}[htb]
\begin{center}
 \begin{tabular}{{c@{\hspace{4mm}}c@{\hspace{8mm}}c}}
  \includegraphics[width=.2\linewidth]{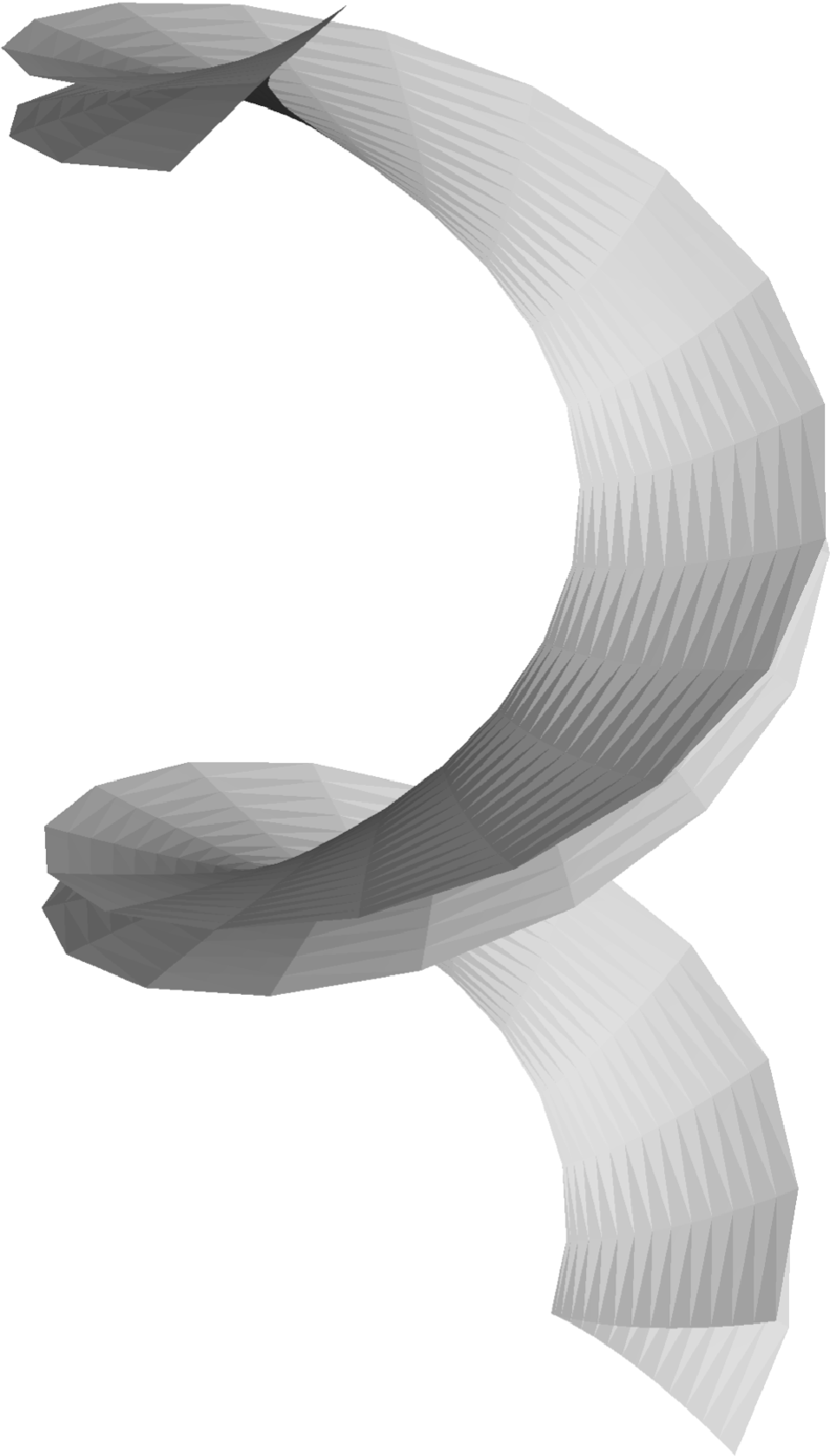}&
  \includegraphics[width=.26\linewidth]{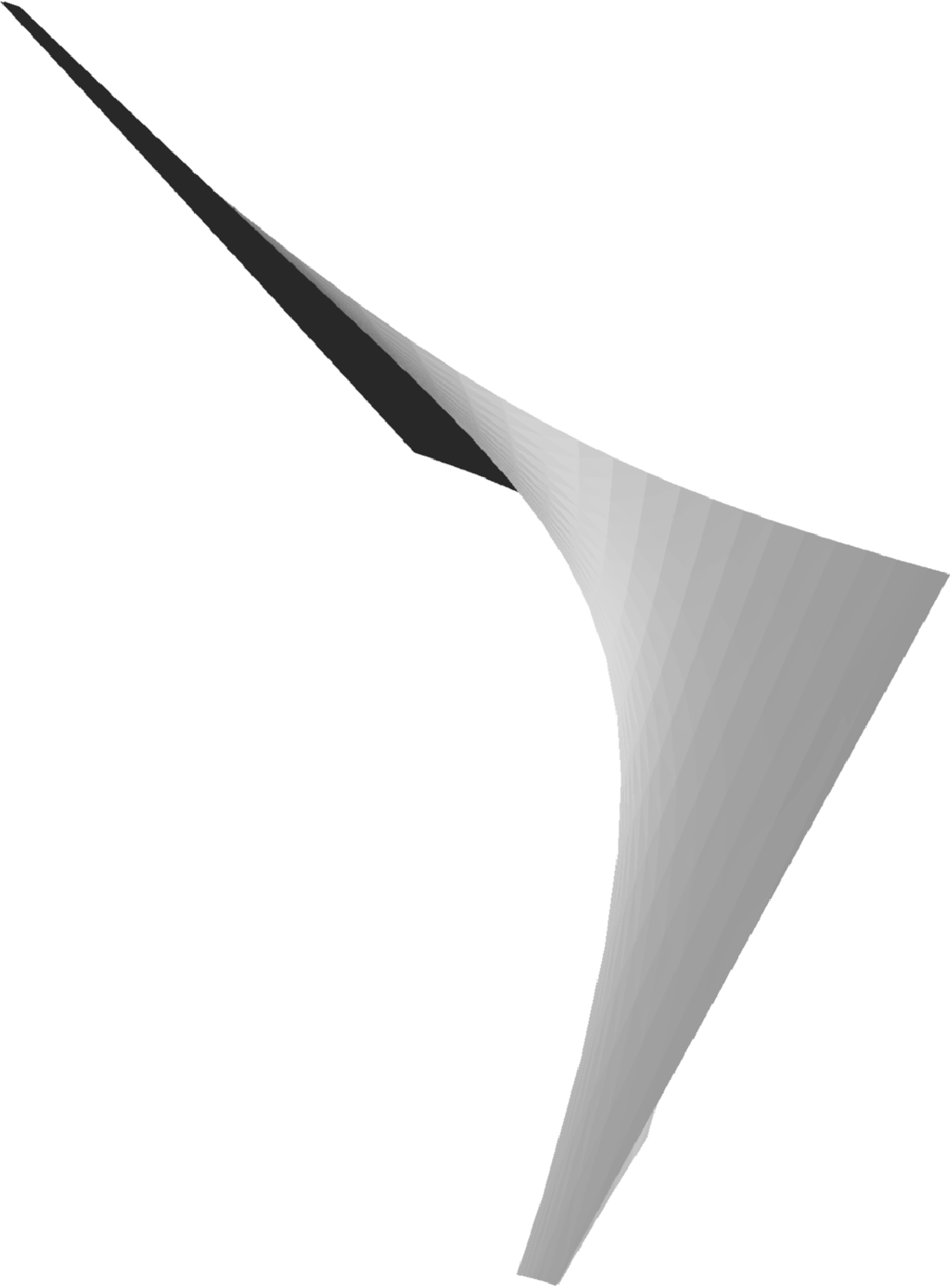}&
  \includegraphics[width=.3\linewidth]{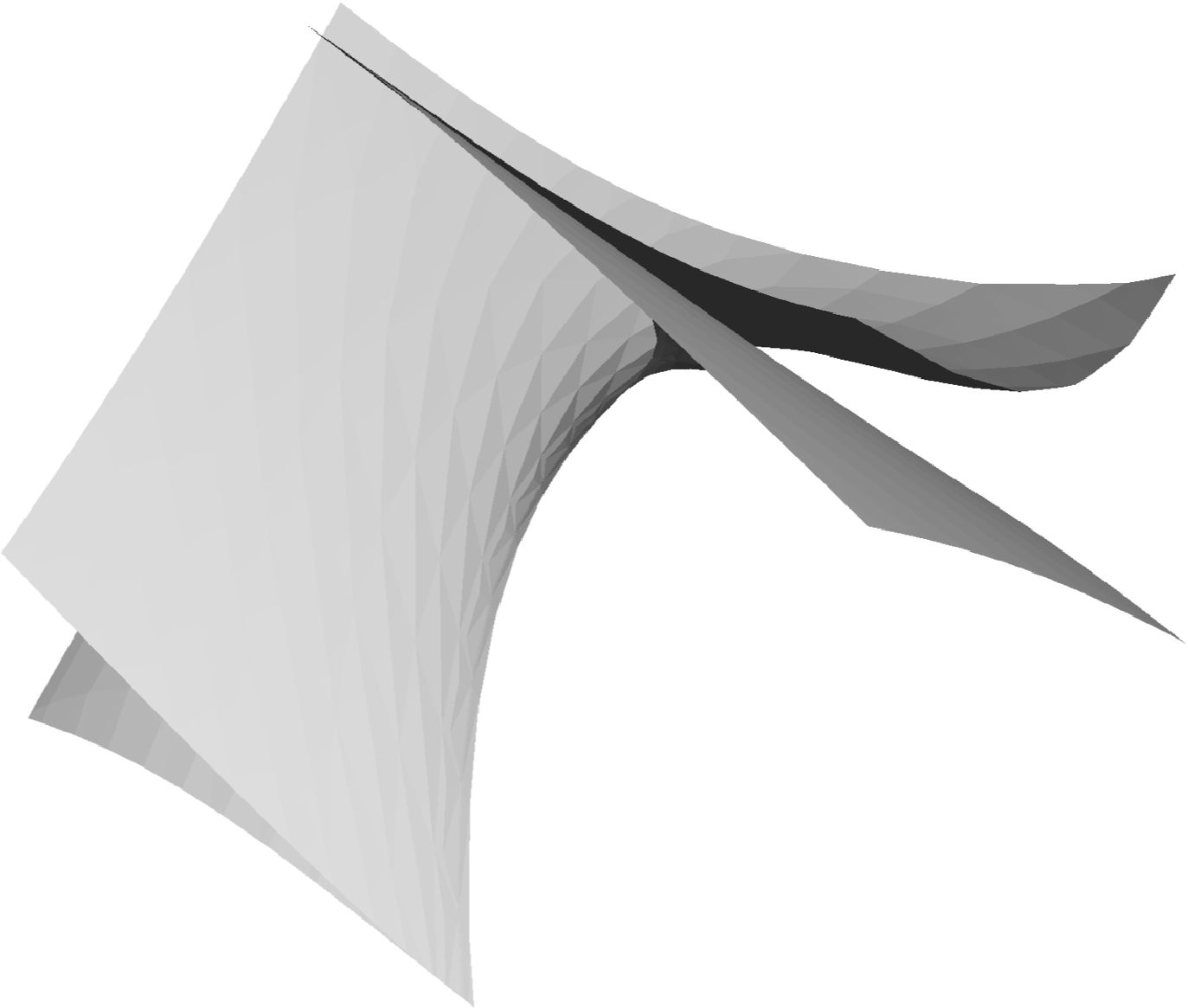} \\
  {\footnotesize  $k=2$} &
  {\footnotesize  $k=-2$} &
  {\footnotesize  $k=-1$}\\
  {\footnotesize  ($f_{{T}}(u,v)$ given in \eqref{eq:t-conj}) } &
  {\footnotesize  ($f_{{S}}(u,v)$ given in \eqref{eq:s-conj}) } &
  {\footnotesize  ($f_{{L}}(u,v)$ given in \eqref{eq:l-conj})}
 \end{tabular}
 \caption{Spacelike helicoidal CMC surfaces $(H=1/2)$
 having ${5/2}$-cuspidal edges in Lorentz-Minkowski $3$-space $\R^3_1$.
 These surfaces are conjugates of spacelike Delaunay surfaces 
 with timelike axis. See \cite{HKS} for more details.}
\label{fig:generic_sing}
\end{center}
\end{figure}

Here, we consider the case of $f=f_{{T}}(u, v)$ given in \eqref{eq:t-conj}.
Similar computations can be applied 
in the cases of $f_{{S}}$ and $f_{{L}}$ given in 
\eqref{eq:s-conj} and \eqref{eq:l-conj}, respectively.
For simplification, we may assume that $H>0$.

Since $(f_{{T}})_u(0, v)=0$, 
the singular set $S(f_{{T}})$ is given by $S(f_{{T}})=\{u=0\}$
and $\eta=\partial_u$ gives a null vector field.
Since the map $\nu : (\R^2,0) \rightarrow S^2$
defined by
\begin{align*}
\nu=&
\dfrac{1}{ \sqrt{2\Delta } \sqrt{\delta -(k+1) u^2}}
\bigg(
\sqrt{\delta \Delta },
- \sqrt{2(k+1)} u^3 \cos \phi -\sqrt{\delta } (k-1) \sin \phi ,\\
&\hspace{50mm}
 - \sqrt{2(k+1)} u^3 \sin \phi + \sqrt{\delta } (k-1) \cos \phi 
\bigg)
\end{align*}
is a unit normal vector field along $f_{{T}}$ (cf.\ Section \ref{sec:frontal}),
$f_{{T}}$ is a frontal.
Then we can check that 
$\eta\lambda(0,v)= - 1/(2 H^2 \sqrt{k+1})\,(\neq0)$
holds,
where $\lambda$ 
is the signed area density function 
(cf.\ \eqref{eq:area-density}).
Thus, we have that $f_{{T}}$ satisfies \ref{itm:etal} 
in Proposition \ref{prop:criteria}.
Set $\xi(u,v):=\partial_v$ and
$$
  \tilde{\eta}(u,v) := \partial_u 
  -\frac{2{\rm sign}(k-1)}{(k-1)^2}u^2 \partial_v
$$
(cf.\ {\eqref{eq:eta-alpha}}).
Then we can check that
$\inner{\xi f_{{T}}(0,v)}{\tilde{\eta}^2 f_{{T}}(0,v)}=0$
and $\tilde{\eta}^3 f_{{T}}(0,v)=0$.
Hence, $f_{{T}}$ satisfies \ref{itm:no23} 
in Proposition \ref{prop:criteria}.
Moreover, the constant $l$ is $0$.
Then, by a straightforward calculation, we have
$$
  \det\left(\xi f_{{T}},\tilde\eta^2 f_{{T}}, \tilde\eta^5f \right)
  (0,v)
  =-\frac{24}{H^2 |k-1|^3}\,(\neq0).
$$
Therefore, 
$f_{{T}}$ satisfies \ref{itm:25new} 
in Proposition \ref{prop:criteria},
and hence $f_{{T}}$ has ${5/2}$-cuspidal edges 
along $\gamma(v)=(0,v)$.
The invariants are calculated as
$$
  r_c(0,v) = \frac{72 H^{3/2} \sqrt{k+1}}{\sqrt{|k-1|}},\qquad
  r_b(0,v) = 0.
$$
Similarly, in the case of $k<-1$,
the invariants of $f_{{S}}$ given in \eqref{eq:s-conj} are calculated as
\begin{equation*}
  r_c(0,v) = \frac{72 H^{3/2} \sqrt{-k-1}}
                 { \sqrt{1-k} \cosh \left(\frac{\sqrt{-k-1} v}{\sqrt{2}}\right) },\quad
  r_b(0,v) = 6 \sqrt{2} H \frac{(1+k) \sinh \left(\frac{\sqrt{-k-1} v
                }{\sqrt{2}}\right)}{(1-k) \cosh ^2\left(\frac{\sqrt{-k-1} v}{\sqrt{2}}\right)},
\end{equation*}
and in the case of $k=-1$, 
the invariants of $f_{{L}}$ given in \eqref{eq:l-conj} are
calculated as
\begin{equation*}
  r_c(0,v) = -\frac{72 \sqrt{H} }{1+v^2},\qquad
  r_b(0,v) = \frac{6 \sqrt{2} v}{H \left(1+v^2\right)^2}.
\end{equation*}

\section{Intrinsicity and extrinsicity of invariants}
\label{sec:intrinsic}
Let $f:(\R^2,0)\to (\R^3,0)$ be a map-germ.
The {\it induced metric\/} or the {\it first 
fundamental form\/} of $f$ is 
the metric on $(\R^2,0)$ defined by
$f^*\inner{~}{~}$.
A function $I:(\R^2,0)\to \R$,
or $I:S(f)\to \R$, is an {\it invariant\/}
if $I$ does not depend on the choice of coordinate
system on the source.
An invariant $I:(\R^2,0)\to \R$,
or $I:S(f)\to \R$, is {\it intrinsic\/}
if it can be represented by a $C^\infty$ function
of $E,F,G$ and {their derivatives}, where 
$$
E=\inner{f_u}{f_u},\quad
F=\inner{f_u}{f_v},\quad
G=\inner{f_v}{f_v},
$$
and $(u,v)$ is a coordinate defined in terms of
the first fundamental form
$f^*\inner{~}{~}$.
An invariant $I:(\R^2,0)\to \R$,
or $I:S(f)\to \R$, is {\it extrinsic\/}
if there exists a map $\tilde f$
such that the first fundamental form of $\tilde f$ is
the same as $f$,
but $I$ does not coincide.
In \cite{MSUY, HHNSUY}, it is determined whether 
some invariants of cuspidal edges
are intrinsic or extrinsic 
(cf.\ \cite{intcr} for invariants of cross caps).
In this section, we show $r_b$ is extrinsic.

\subsection{Intrinsic criterion for ${5/2}$-cuspidal edges}
Let $f:(\R^2,0)\to(\R^3,0)$ be a frontal-germ and $0$
a non-degenerate singular point.
Here, we shall show that 
the $\A$-equivalence class of ${5/2}$-cuspidal edges
can be determined intrinsically
{among frontal-germs with non-zero limiting normal curvature 
$\kappa_\nu\neq0$ 
(Theorem \ref{thm:intrinsic_1st}, Corollary \ref{cor:intrinsic-criterion})}. 

\begin{definition}\label{def:normally-adjusted}
Let $f:(\R^2,0)\to(\R^3,0)$ be a frontal-germ
such that $0$ is a singular point of the first kind.
A coordinate system $(u,v)$ around $0$ 
is called {\it adjusted at $0$} if $f_v(0,0)=0$.
A coordinate system $(u,v)$ which is adjusted at $0$
is called {\it normally-adjusted at $0$} 
if $(u,v)$ is compatible with the orientation of $(\R^2,0)$,
$E(0,0)=1$ and $\lambda_v(0,0)=1$.
\end{definition}

The existence of such a normally-adjusted coordinate system 
can be verified by the existence of 
{\it normalized strongly adapted coordinate systems}\footnote{%
A coordinate system $(u,v)$ centered at $(0,0)$ is
called {\it normalized strongly adapted} if
the singular set is given by the $u$-axis, 
$\partial_v$ gives the null vector field along the $u$-axis,
$f_v(u,0)=0$, $|f_u(u,0)|=|f_{vv}(u,0)|=1$,
$\inner{f_u(u,0)}{f_{vv}(u,0)}=0$
and $\inner{f_u(u,v)}{f_v(u,v)}=0$ hold.
} \cite[Definition 2.24, Proposition 2.25]{HHNSUY}
(cf.\ \cite[Lemma 3.2]{SUY} and \cite[Definition 3.7]{MSUY}).

It was proved in \cite[Corollary 3.14]{MSUY}
that the Gaussian curvature $K$ and the mean curvature $H$ can be
extended smoothly across ${5/2}$-cuspidal edges.
Then, we set
\begin{equation}\label{eq:kappa-ni-tilde}
  H_{\eta}:=H_{v}(0,0),\qquad
  K_{\eta}:=K_{v}(0,0),
\end{equation}
where $(u,v)$ is a coordinate system normally-adjusted at $0$.
We call $K_{\eta}$ (respectively, $H_{\eta}$) 
the {\it null-derivative Gaussian curvature} 
(respectively, the {\it null-derivative mean curvature}\/)
of ${5/2}$-cuspidal edge at $0$.
We shall prove that the definitions of
null-derivative Gaussian and mean curvature
do not depend on the choice of 
normally-adjusted coordinate systems, {as follows}.

\begin{lemma}\label{lem:n-adjust}
If two coordinate systems $(u,v)$ and $(U,V)$
are normally-adjusted at $0$,
then 
\begin{equation}\label{eq:normally-adjusted-change}
  U_u=1,\qquad
  U_v=0,\qquad
  V_v=1
\end{equation}
holds at $(0,0)$.
Moreover, 
the definitions of 
null-derivative Gaussian and mean curvatures
$H_{\eta}$, $K_{\eta}$
are independent of the choice of 
the coordinate system normally-adjusted at $0$.
\end{lemma}

\begin{proof}
Since $f_v=f_V=0$ at $(0,0)$,
$$
  f_v = U_v f_U + V_v f_V = U_v f_U  
$$
yields $U_v(0,0)=0$.
Since $(u,v) \mapsto (U,V)$ is orientation-preserving,
$J:=U_uV_v - U_vV_u$ is positive-valued.
In particular, $J(0,0)=U_u(0,0)V_v(0,0)>0$ holds.
Setting $\lambda:=\det (f_u,f_v,\nu)$ and $\Lambda:=\det (f_U,f_V,\nu)$, 
we have 
$\lambda=J\Lambda$.
Then
$$
  \lambda_v 
  = J_v\Lambda+J\Lambda_v
  = J_v\Lambda+J(\Lambda_U U_v+\Lambda_V V_v)
$$
holds, and evaluating this at $(0,0)$ we have
\begin{equation}\label{eq:lambda_v}
  1 = U_u(0,0)V_v^2(0,0),
\end{equation}
which yields $U_u(0,0)>0$.
Since $J=U_uV_v>0$ at $(0,0)$,
$V_v(0,0)>0$ holds.
Moreover, by
$
  1 = E
  =\inner{f_u}{f_u}
  =U_u^2\inner{f_U}{f_U}
  =U_u^2
$
at $(0,0)$,
we have $U_u(0,0)=1$.
Substituting this into \eqref{eq:lambda_v}, $V_v(0,0)=1$ holds.
Hence we have \eqref{eq:normally-adjusted-change}.
Moreover, then
$$
 \frac{\partial}{\partial v}
  = U_v \frac{\partial}{\partial U}
     + V_v \frac{\partial}{\partial V}
  = \frac{\partial}{\partial V}
$$
holds at $(0,0)$.
In particular, 
the definition of $H_{\eta}$, $K_{\eta}$
as in \eqref{eq:kappa-ni-tilde}
is independent of choice of 
the coordinate system normally-adjusted at $0$.
\end{proof}

Since the Gaussian curvature $K$ and
the definition of normally-adjusted coordinate systems 
are intrinsic, the null-derivative Gaussian curvature $K_{\eta}$ 
is an intrinsic invariant for ${5/2}$-cuspidal edges.
Now, we shall check the relationships amongst
$K$, $H$, $K_{\eta}$, $H_{\eta}$ and other invariants.

\begin{lemma}\label{lem:K}
Let $f:(\R^2,0)\to(\R^3,0)$ be a germ of a ${5/2}$-cuspidal edge.
Then
the Gaussian curvature $K$ and
the mean curvature $H$ of $f$
satisfy
\begin{align}
  \label{eq:K}
  &K = \frac{1}{3} \kappa_\nu r_b - \kappa_t^2,\\
  \label{eq:H}
  &H = \frac{1}{2} \kappa_\nu + \frac{1}{6}r_b
\end{align}
along the singular set, respectively.
Moreover, 
the null-derivative Gaussian curvature $K_{\eta}$
and the null-derivative mean curvature $H_{\eta}$
of $f$ satisfy
\begin{align}
  \label{eq:Kv}
  &K_{\eta} = \frac1{24}r_\Pi,\\
  \label{eq:Hv}
  &H_{\eta} = \frac1{48}r_c
\end{align}
along the singular set, respectively.
\end{lemma}

\begin{proof}
By Proposition \ref{prop:normalform},
without loss of generality, we may assume that
$f$ is given by the form in \eqref{eq:normalform}.
A direct calculation yields
$$
\kappa_t(0) = 2 b_{12},\quad 
r_b(0)=b_{04},\quad 
\kappa_\nu(0)=b_{20},\quad
r_c(0)=3 b_{05},\quad
r_\Pi(0)=3 b_{20} b_{05}
$$
and 
$$
K = \frac1{3}b_{20} b_{04} - 4 b_{12}^2,\quad
H = \frac1{2}b_{20} +\frac1{6}b_{04}
$$
hold at $(0,0)$.
Hence, \eqref{eq:K} and \eqref{eq:H} hold.
On the other hand,
since the coordinate system $(u,v)$ of
$f(u,v)$ given by the form in \eqref{eq:normalform}
is normally-adjusted at $(0,0)$,
we have 
$H_v = H_{\eta}$ and
$K_v = K_{\eta}$ at $(0,0)$.
By a direct calculation, we have that
$$
H_v(0,0) = \frac1{16}b_{05},\qquad
K_v(0,0) = \frac1{8}b_{20} b_{05},
$$
and hence, \eqref{eq:Kv} and \eqref{eq:Hv} hold.
\end{proof}

\begin{theorem}\label{thm:intrinsic_1st}
For ${5/2}$-cuspidal edges,
the secondary product curvature $r_\Pi$ is an intrinsic invariant.
\end{theorem}

\begin{proof}
By \eqref{eq:Kv} in Lemma \ref{lem:K}
and the fact that $K_{\eta}$ is intrinsic,
$r_\Pi$ is intrinsic as well.
\end{proof}

{
Let $f : (\R^2,0)\to (\R^3,0)$ be 
a frontal-germ such that
$0$ is a non-degenerate singular point.
If $\kappa_\nu(0)\neq0$,
then $f$ is called \emph{non-$\nu$-flat}.
The following corollary implies that 
the $\A$-equivalence class of ${5/2}$-cuspidal edges
can be determined intrinsically
amongst non-$\nu$-flat frontal-germs.
}

\begin{corollary}\label{cor:intrinsic-criterion}
Let $f : (\R^2,0)\to (\R^3,0)$ be a frontal-germ
such that $0$ is a singular point of the first kind.
Assume that $f$ is non-$\nu$-flat.
Then,
$f$ at $0$ is a ${5/2}$-cuspidal edge
if and only if 
$\kappa_\Pi=0$ along $S(f)$ and $r_\Pi(0)\neq0$.
\end{corollary}

\begin{proof}
By the definitions of $\kappa_c$ given in \eqref{eq:kappa-c},
$r_c$ given in \eqref{eq:kappa-c-ni} and
the criterion (Proposition \ref{prop:criteria}),
$f$ at $0$ is a ${5/2}$-cuspidal edge
if and only if 
$\kappa_c=0$ along $S(f)$ and $r_c(0)\neq0$.
Therefore, 
imposing the non-$\nu$-flatness $\kappa_\nu\neq0$,
we have that
$0$ is a non-$\nu$-flat ${5/2}$-cuspidal edge
if and only if 
$\kappa_\Pi=0$ along $S(f)$ and $r_\Pi(0)\neq0$.
\end{proof}

Following Corollary \ref{cor:intrinsic-criterion}, 
we give a definition of {\it intrinsic ${5/2}$-cuspidal edges}
for singular points of a certain metric, called the Kossowski metric
in Section \ref{sec:Kossowski}
(cf.~Definition \ref{def:intrinsic-ramphoid}).

\subsection{Isometric deformations of ${5/2}$-cuspidal edges}
The following fact is a direct conclusion 
of \cite[Theorem B]{HNUY}:

\begin{fact}[\cite{HNUY}]\label{fact:B}
Let $f : (\R^2,0)\to (\R^3,0)$ be an analytic frontal-germ
such that $0$ is a singular point of the first kind,
and $\gamma : (\R,0) \to (\R^2,0)$ a singular curve.
Assume that $f$ has non-vanishing limiting normal curvature. 
Then, for given analytic functions germs $\omega(t)$, $\tau(t)$ at $t=0$,
there exists an analytic frontal-germ $g=g_{\omega,\tau}$ 
such that
\begin{enumerate}
\item\label{item:B1}
the first fundamental form of $g_{\omega,\tau}$ coincides with that of $f$,
\item\label{item:B2}
the limiting normal curvature function of $g_{\omega,\tau}$ 
along $\gamma$ coincides with $e^{\omega(t)}$
for a suitable choice of a unit normal vector field, and
\item\label{item:B3}
$\tau(t)$ gives the torsion function of $\hat{\gamma}_g(t)$,
where $\hat{\gamma}_g(t):=g\circ \gamma(t)$.
\end{enumerate}
The possibilities for congruence classes of such a $g$ 
are at most two unless $\tau$ vanishes identically.
On the other hand, if $\tau$ vanishes identically
$($i.e., $\hat{\gamma}_g$ is a planar curve$)$,
then the congruence class of $g$ is uniquely determined.
\end{fact}

Using Fact \ref{fact:B},
we shall prove the following,
which is an analog of a result of 
\cite[Theorem A]{NUY}
and \cite[Corollary D]{HNUY}.

\begin{theorem}[Isometric deformation of ${5/2}$-cuspidal edges]
\label{thm:isom-defo}
Let $f:(\R^2,0)\to(\R^3,0)$ be a germ of an analytic ${5/2}$-cuspidal edge
with non-vanishing limiting normal curvature,
and let $\kappa_s(t)$ be the singular curvature function 
along the singular curve $\gamma(t)$.
Take a germ of an analytic regular space curve $\sigma(t)$
such that its curvature function $\kappa(t)$ satisfies
$$
  \kappa > |\kappa_s|
$$
at $0$.
Then, there exists a germ of an analytic ${5/2}$-cuspidal edge 
$g_{\sigma}:(\R^2,0)\to(\R^3,0)$
with non-vanishing limiting normal curvature
such that
\begin{enumerate}
\item
the first fundamental form of $g_{\sigma}$ coincides with that of $f$,
\item
the singular image $g_{\sigma}\circ \gamma$ 
coincides with $\sigma$.
\end{enumerate}
The possibilities for congruence classes of such a $g_{\sigma}$ 
are at most two unless $\tau$ vanishes identically.
On the other hand, if $\tau$ vanishes identically
$($i.e., $\sigma$ is a planar curve$)$,
then the congruence class of $g_{\sigma}$ is uniquely determined.
\end{theorem}

\begin{proof}
Set $\omega(t)$ as
$$
  \omega(t) := \frac1{2} \log \left( \kappa(t)^2 - \kappa_s(t)^2 \right).
$$
Let $\tau(t)$ be the torsion function of $\sigma(t)$.
By Fact \ref{fact:B}, 
there exists an analytic frontal-germ 
$g_{\sigma}:=g_{\omega,\tau}:(\R^2,0)\to(\R^3,0)$
such that the items \ref{item:B1}--\ref{item:B3} in Fact \ref{fact:B} hold.
Thus, it suffices to show that
$g_{\sigma}$ has a ${5/2}$-cuspidal edge at $0$.
Since
the first fundamental form of $f$ coincides with that of $g_{\sigma}$, 
the product curvature $\kappa_\Pi$ 
and the secondary product curvature $r_\Pi$
of $f$ coincide with
those of $g_{\sigma}$, respectively.
Therefore, by Corollary \ref{cor:intrinsic-criterion},
we have that $g_{\sigma}:(\R^2,0)\to(\R^3,0)$ has a ${5/2}$-cuspidal edge at $0$.
\end{proof}

In \cite{HNUY}, the following is also proved.

\begin{fact}[{\cite[Corollary E]{HNUY}}]
\label{fact:E}
Let $f_0$, $f_1$ be two analytic frontal germs with 
non-degenerate singularities
whose limiting normal curvatures do not vanish.
Suppose that they are isometric to each other.
Then there exists {a continuous} 1-parameter family
of frontal germs $g_t$ $(0\le t \le 1)$
satisfying the following properties\/{\rm :}
\begin{enumerate}
\item\label{item:E1} $g_0=f_0$ and $g_1=f_1$,
\item\label{item:E2} $g_t$ is isometric to $g_0$,
\item\label{item:E3} the limiting normal curvature of each $g_t$
does not vanish.
\end{enumerate}
Moreover, if both $f_0$ and $f_1$
are germs of cuspidal edges, swallowtails or  cuspidal cross caps,
then so are $g_t$ for $0\le t\le 1$.
\end{fact}

By this fact and Corollary \ref{cor:intrinsic-criterion},
we also have the following result analogous to \cite[Corollary E]{HNUY}.

\begin{corollary}
\label{cor:interpolating}
Let $f_0$, $f_1$ be two analytic germs of ${5/2}$-cuspidal edges
whose limiting normal curvatures do not vanish.
Suppose that they are isometric to each other.
Then there exists {a continuous} 1-parameter family
of germs $g_t$ of ${5/2}$-cuspidal edges $(0\le t \le 1)$
satisfying the following properties\/{\rm :}
\begin{enumerate}
\item $g_0=f_0$ and $g_1=f_1$,
\item $g_t$ is isometric to $g_0$,
\item the limiting normal curvature of each $g_t$
does not vanish.
\end{enumerate}
\end{corollary}

\begin{proof}
By Fact \ref{fact:E}, 
there exists {a continuous} 1-parameter family
of frontal germs $g_t$ $(0\le t \le 1)$
such that the items \ref{item:E1}--\ref{item:E3} in Fact \ref{fact:E} hold.
Since the limiting normal curvature of each $g_t$
does not vanish and $g_t$ is isometric to $g_0$ for each $t\in [0,1]$,
Corollary \ref{cor:intrinsic-criterion} yields that
$g_t$ has ${5/2}$-cuspidal edges.
Hence, 
the family $\{g_t\}_{t\in [0,1]}$ is the desired one.
\end{proof}

\subsection{Extrinsity of invariants}
Let $f:(\R^2,0)\to(\R^3,0)$ be a germ of 
a non-$\nu$-flat ${5/2}$-cuspidal edge,
and let $\gamma : (\R,0)\to(\R^2,0)$ be a germ of a singular curve of $f$.
Let $\hat{\gamma}$ be the regular curve in $\R^3$ 
given by $\hat{\gamma}:=f \circ \gamma$, with {arclength} parameter $t$.
Set $\vect{e}(t):= \hat{\gamma}'(t)$
and  $\vect{b}(t) := - \vect{e}(t) \times \hat{\nu}(t)$,
where $\hat{\nu}(t):= \nu(\gamma(t))$.
Then $\{ \vect{e}, \vect{b}, \hat{\nu} \}$
is an orthonormal frame along $\gamma$.
{Remark that, in general, $\vect{b}$ may not coincide with 
the binormal vector field of $\hat{\gamma}$ as a space curve.}
Moreover we have
\begin{equation}\label{eq:Frenet-Serret}
  \hat\gamma'' = \kappa_s \vect{b} + \kappa_\nu \hat{\nu}, \qquad
  \vect{b}' = -\kappa_s \vect{e} + \kappa_t \hat{\nu}, \qquad
  {\hat\nu' = -\kappa_\nu \vect{e} - \kappa_t \vect{b}}.
\end{equation}
Let $\kappa$, $\tau$ be the curvature and torsion functions
of $\hat \gamma$, respectively.
Substituting \eqref{eq:Frenet-Serret}
into $\kappa^2\tau=\det(\hat\gamma', \hat\gamma'', \hat\gamma''')$,
we have the following.

\begin{lemma}\label{lem:geodesic}
It holds that
$$
  \kappa = \sqrt{\kappa_s^2+\kappa_\nu^2},\qquad
  \tau = 
   \frac{\kappa_s' \kappa_\nu - \kappa_s \kappa_\nu'}{\kappa^2} - \kappa_t.
$$
In particular, if $\kappa_s(t)=0$ along $\gamma(t)$, 
then $\kappa_t(t)=-\tau(t)$ holds.
\end{lemma}

As a corollary of Theorem \ref{thm:isom-defo},
we prove the extrinsicity of 
the limiting normal curvature $\kappa_\nu$ 
(Corollary \ref{cor:extrinsic-knu}),
the cuspidal torsion $\kappa_t$
(Corollary \ref{cor:extrinsic-kt}),
and 
the bias $r_b$
(Corollary \ref{cor:Bias-extrinsic}).
We remark that the proof of Corollary \ref{cor:extrinsic-knu}
is analogous to that of \cite[Corollary D]{NUY}.

\begin{corollary}\label{cor:extrinsic-knu}
For ${5/2}$-cuspidal edges,
the limiting normal curvature $\kappa_\nu$
is an extrinsic invariant.
\end{corollary}

\begin{proof}
Let us take a real-analytic germ of a
non-$\nu$-flat ${5/2}$-cuspidal edge $f:(\R^2,0)\to(\R^3,0)$.
Denote by $\kappa(t)$ and $\tau(t)$ 
the curvature and torsion of $\hat{\gamma}(t):=f(\gamma(t))$,
respectively.
By Fact \ref{fact:B}, for a given analytic function $\omega(t)$,
there exists a non-$\nu$-flat real-analytic frontal-germ 
$g_{\omega,\tau}$ such that 
$g_{\omega,\tau}$ is isometric to $f$,
the limiting normal curvature function of $g_{\omega,\tau}$ is $e^{\omega(t)}$,
and $\tau(t)$ gives the torsion function of $g_{\omega,\tau}(\gamma(t))$.
Moreover, by Corollary \ref{cor:intrinsic-criterion},
$g_{\omega,\tau}$ is a ${5/2}$-cuspidal edge.
Since we can choose $\omega(t)$ arbitrarily,
the limiting normal curvature is extrinsic.
\end{proof}

\begin{corollary}\label{cor:extrinsic-kt}
For ${5/2}$-cuspidal edges,
the cuspidal torsion $\kappa_t$ is an extrinsic invariant.
\end{corollary}

\begin{proof}
Let us take a real-analytic germ of a
non-$\nu$-flat ${5/2}$-cuspidal edge $f:(\R^2,0)\to(\R^3,0)$
satisfying $\kappa_s\equiv0$ along the singular curve $\gamma(t)$.
Denote by $\kappa(t)$ and $\tau(t)$ 
the curvature and torsion of $\hat{\gamma}(t):=f(\gamma(t))$,
respectively.
By Lemma \ref{lem:geodesic}, $\tau(t) = -\kappa_t(t)$.
Take an arbitrary analytic function $\tilde{\tau}(t)$.
Then, by the fundamental theorem of space curves, 
there exists an analytic regular space curve $\sigma(t)$ in $\R^3$
whose curvature and torsion functions 
are given by $\kappa(t)$ and $\tilde{\tau}(t)$, respectively.
Applying Theorem \ref{thm:isom-defo} to $\sigma(t)$,
there exists a real-analytic germ of a
non-$\nu$-flat ${5/2}$-cuspidal edge $g_{\sigma}:(\R^2,0)\to(\R^3,0)$
such that 
$g_{\sigma}$ is isometric to $f$
and $\sigma$ gives the image of the singular set of $g_{\sigma}$.
Since $\kappa_s\equiv0$ along the singular curve $\gamma(t)$,
Lemma \ref{lem:geodesic} yields that
the cuspidal torsion of $g_{\sigma}$ is $-\tilde{\tau}(t)$.
Since we can choose $\tilde{\tau}(t)$ arbitrarily,
the cuspidal torsion is extrinsic.
\end{proof}

We remark that 
an analytic non-$\nu$-flat ${5/2}$-cuspidal edge $f$
satisfying $\kappa_s\equiv0$ along $S(f)$
exists.
In fact, 
by rotating the plane curve $(x(t),z(t)):=(1+t^5,t^2)$ with respect to 
the $z$-axis,
we have such an example.

\begin{corollary}\label{cor:Bias-extrinsic}
For ${5/2}$-cuspidal edges,
the bias $r_b$ is an extrinsic invariant.
\end{corollary}

\begin{proof}
Let us take a non-$\nu$-flat real-analytic ${5/2}$-cuspidal edge 
satisfying $\kappa_s\equiv0$ along the singular curve $\gamma(t)$.
Moreover, assuming $\tau\equiv0$, 
then by Lemma \ref{lem:geodesic},
it holds that $\kappa_t\equiv0$.
By Lemma \ref{lem:K},
\[
  K(\gamma(t)) =\frac{1}{3}r_b(t) \kappa_\nu(t).
\]
Let $k \geq0$ be a non-negative real number.
Since $\kappa_\nu\neq0$, 
then by Theorem \ref{thm:isom-defo},
there exists a family $\{g^k\}_{k \geq0}$
of real-analytic germs of ${5/2}$-cuspidal edges 
such that, for each $k\geq0$,
$g^k$ is non-$\nu$-flat,
$g^k$ has the same first fundamental form of $f$,
and 
the curvature function $\kappa^k(t)$ of $g^k(\gamma(t))$ 
is given by $\kappa^k(t)=\kappa(t)+k$ and the torsion is $0$.
Since $f$ and $g^k$ have the same first fundamental form for each $k\geq0$,
the singular curvature $\kappa_s^k(t)$ of $g^k$ 
vanishes identically along $\gamma(t)$.
Thus the limiting normal curvature of ${g^k}$
is ${\kappa_\nu^k(t)=\kappa(t)+k}\,(>0)$.
Hence the bias ${r_b^k}$ for ${g^k}$ is given by
\[
  {r_b^k(t)} 
  = 3\frac{K(\gamma(t))}{\kappa_\nu^k}
  = 3\frac{K(\gamma(t))}{\kappa(t)+k}.
\]
In particular, the bias is extrinsic.
\end{proof}

\begin{remark}
The secondary cuspidal curvature $r_c$ is also extrinsic,
since $r_\Pi$ is intrinsic (Theorem \ref{thm:intrinsic_1st}),
$\kappa_\nu$ is extrinsic (Corollary \ref{cor:extrinsic-knu}),
and $r_c$ is written as $r_c=r_\Pi/\kappa_\nu$ when $\kappa_\nu\neq0$.
Moreover, the product $\kappa_\nu r_b$ is also extrinsic,
since $\kappa_\nu r_b=3(K+\kappa_t^2)$ holds by \eqref{eq:K}
and $\kappa_t$ is extrinsic (Corollary \ref{cor:extrinsic-kt}).
Furthermore,
by a proof similar to that of Corollary \ref{cor:extrinsic-kt},
we can prove that the cuspidal torsion $\kappa_t$ 
for cuspidal edges is also extrinsic.
\end{remark}

\subsection{Summary of intrinsicity and extrinsicity}
We can summarize the intrinsicity and extrinsicity as follows.
As seen in Section \ref{sec:rhamphoidcusp}, 
the corresponding invariant of
the bias of cusps $r_b$ does not exist for cuspidal edges.

\begin{table}[!h]
\centering
\begin{tabular}{c|ccccc}
\hline
invariants&$\kappa_s$&$\kappa_\nu$&$\kappa_t$&
$\kappa_c$&$\kappa_\Pi=\kappa_c\kappa_\nu$\\
int/ext&intrinsic&extrinsic&extrinsic&extrinsic&intrinsic\\
\hline
\end{tabular}
\label{table:CE}
\caption{Intrinsicity and extrinsicity for cuspidal edges.}
\bigskip

\begin{tabular}{c|cccccccc}
\hline
invariants&$\kappa_s$&
$\kappa_\nu$&$\kappa_t$&$r_b$&$r_c$
&
$\kappa_\nu r_b$
&
$\displaystyle 
\kappa_\nu r_b - 3\kappa_t^2$&
$\displaystyle  r_\Pi = \kappa_\nu r_c$\\
int/ext&int&ext&ext&ext&ext
&ext&int&int
\\
\hline
\end{tabular}
\caption{Intrinsicity (int) and extrinsicity (ext) for ${5/2}$-cuspidal edges.
Here, we remark that 
the intrinsicity of the invariant in the seventh slot 
can be verified by the identity
$\kappa_\nu r_b - 3\kappa_t^2=3K$ (cf.\ \eqref{eq:K}).
With respect to the eighth slot, 
see Theorem \ref{thm:intrinsic_1st}.}
\label{table:RCE}
\end{table}

\section{Isometric realizations of intrinsic ${5/2}$-cuspidal edges}
\label{sec:Kossowski}

In this section, we deal with ${5/2}$-cuspidal edge singularities 
without ambient spaces. 
We give a definition of intrinsic ${5/2}$-cuspidal edges 
for singular points of Kossowski metrics,
and prove the existence of their isometric realizations
(Theorem \ref{thm:isom-realization})
as in \cite{NUY} and \cite{HNUY}.

First, we briefly introduce the basic properties of 
Kossowski metrics.
Further systematic treatments of Kossowski metrics 
are given in \cite{HHNSUY, SUY3, HNUY}.
Let $ds^2$ be a germ of a positive semi-definite metric
on $(\R^2,0)$.
Assume that $0$ is a singular point of $ds^2$,
that is, $ds^2$ is not positive-definite at $0$.
Denote by $S(ds^2)$ the set of singular points.
A non-zero tangent vector $\vect{v}$ at $0$
is called a {\it null vector} at $0$ if  
$ds^2(\vect{v},\vect{x})=0$ holds for every tangent vector $\vect{x}$ at $0$.
A local coordinate neighborhood $(U;u,v)$ 
is called {\it adjusted} at $0$ if
$\partial_v=\partial/\partial v$ gives a null vector at $(0,0)$.

If $(U,u,v)$ is a local coordinate neighborhood 
adjusted at $0$, then 
$F=G=0$ holds at $(0,0)$, where
\begin{equation}\label{eq:ds2}
  ds^2=E\,du^2+2F\,du\,dv+G\,dv^2.
\end{equation}
A singular point $0$ is called {\it admissible\/}
if there exists an local coordinate
neighborhood $(U;u,v)$ adjusted at $0$ 
such that
$
  E_v=2F_u,\,
  G_u=G_v=0
$
hold at $(0,0)$.

\begin{definition}[Kossowski metric]
If each singular point is admissible, 
and there exists a smooth function $\lambda$ 
defined on a neighborhood $(U;u,v)$ of $0$ 
such that
$$
EG-F^2=\lambda^2
$$
on $U$, and $d\lambda\neq0$ holds at $(0,0)$,
then $ds^2$ is called a (germ of a) {\it Kossowski metric},
where $E,F,G$ are smooth functions on $U$
satisfying \eqref{eq:ds2}.
Moreover, if we can choose $E$, $F$, $G$, and $\lambda$
to be analytic functions, then the Kossowski metric is called {\it analytic}.
\end{definition}

As shown in \cite{HHNSUY},
the first fundamental form of
a frontal-germ $f:(\R^2,0)\to (\R^3,0)$
whose singular points are all non-degenerate
is a Kossowski metric.

Let $ds^2$ be a germ of a Kossowski metric having a singular point at $0$.
By the condition $d\lambda\neq0$ at $(0,0)$,
the implicit function theorem yields that
there exists a regular curve
$\gamma(t)$ $(|t|<\varepsilon)$ in the $uv$-plane 
(called the {\it singular curve})
parametrizing $S(ds^2)$.
Then there exists a smooth non-zero vector field $\eta$
such that $\eta_q$ gives a null vector for each $q\in S(ds^2)$ near $(0,0)$.
We call $\eta$ a {\it null vector field}.

\begin{definition}
If $\eta$ is transversal to $S(ds^2)$ at $0$,
the singular point $0$ is called \emph{type {\rm I}} (or an \emph{$A_2$ point}\/).
\end{definition}

For a Kossowski metric $ds^2$ induced from a frontal-germ $f$,
type {\rm I} singular points of $ds^2$ correspond to 
singular points of the first kind of $f$.

According to \cite[Proposition 2.25]{HHNSUY},
for a type {\rm I} singular point,
there exists a coordinate system $(U;u,v)$ centered at $0$,
such that 
\begin{itemize}
\item the singular set $S(ds^2)$ is given by the $u$-axis,
\item $\partial_v$ gives the null vector field,
\item $F=0$ on $U$, and
\item $
  E(u,0)=1,~
  E_v(u,0)=G_v(u,0)=0,~
  G_{vv}(u,0)=2$
\end{itemize}
hold, where $E$, $F$, $G$ are smooth functions as in \eqref{eq:ds2}.
Such a coordinate system is called
a {\it normalized strongly adapted coordinate system}.
Since $G_{vv}(u,0)=2$ is equivalent to $\lambda_v(u,0)=\pm 1$,
by changing $v\mapsto -v$ if necessary,
we may assume that $\lambda_v(u,0)=1$.
(Hence, in the case of Kossowski metrics induced from frontals in $\R^3$,
the normalized strongly adapted coordinate systems
are normally-adjusted, cf.~Definition \ref{def:normally-adjusted}.)

We shall review the definition of the product curvature 
for type {\rm I} singular points defined in \cite{HHNSUY}.
Let $(U;u,v)$ be a normalized strongly adapted coordinate system
centered at a type {\rm I} singular point $0$.
Denote by $K$ the Gaussian curvature of $ds^2$ 
on $U\setminus\{v=0\}$.
By \cite[Proposition 2.27]{HHNSUY},
$v K(u,v)$ is a smooth function on $U$.
Then 
\[
  \tilde{\kappa}_\Pi:=\lim_{v\to 0}v K(u,v)
\]
does not depend on the choice of 
the normalized strongly adapted coordinate system 
satisfying $\lambda_v(0,0)=1$,
and is called the {\it product curvature}.

Now, assume that $\tilde{\kappa}_\Pi$ vanishes 
along the $u$-axis.
Then, $K$ is a bounded smooth function on $U$,
and 
\[
  \tilde{K}_\eta := \lim_{v\to 0}K_v(u,v)
\]
does not depend on the choice of 
the normalized strongly adapted coordinate system
satisfying $\lambda_v(0,0)=1$.
We call $\tilde{K}_\eta$ 
the {\it secondary product curvature}
or the {\it null-derivative Gaussian curvature}.

\begin{definition}\label{def:intrinsic-ramphoid}
Let $ds^2$ be a germ of a Kossowski metric 
at a type {\rm I} singular point $0$.
If the product curvature $\tilde{\kappa}_\Pi$
vanishes along $S(ds^2)$,
and the secondary product curvature $\tilde{K}_\eta$
does not vanish at $0$,
then the singular point $0$
is called an {\it intrinsic ${5/2}$-cuspidal edge}.
\end{definition}

The following lemma is a direct conclusion of 
Lemma \ref{lem:K} and Corollary \ref{cor:intrinsic-criterion}.

\begin{lemma}\label{lem:intrinsic-ramphoid}
Let $f: (\R^2,0)\to (\R^3,0)$ be a non-$\nu$-flat frontal-germ 
having a singular point $0$ of the first kind.
Denote by $ds^2$ the first fundamental form of $f$.
Then, $f$ at $0$ is a ${5/2}$-cuspidal edge
if and only if 
$0$ is an intrinsic ${5/2}$-cuspidal edge 
{\rm (}as a singular point of the Kossowski metric $ds^2)$.
\end{lemma}

We remark that the assumption of the non-$\nu$-flatness 
cannot be removed,
since there exists 
a cuspidal edge with vanishing limiting normal curvature
such that the corresponding singular points of $ds^2$
are intrinsic ${5/2}$-cuspidal edges.
\begin{example}\label{ex:int-rhamphoid-CE}
{
Let $f: (\R^2,0)\to (\R^3,0)$ be 
a map-germ defined by
$
  f(u,v) = (u,u^2+v^2,v^3+v^4).
$
The first fundamental form $ds^2$ is written as
$$
  ds^2=(4u^2+1)du^2 + 8uv\,du\,dv 
  +v^2 \left(4+v^2(4v+3)^2\right)dv^2.
$$
We can check that
$f$ is a front with a unit normal 
$\nu(u,v)=\hat\lambda^{-1}(2uv(4v+3),-v(4v+3),2)$,
where we set
$\hat\lambda:=\sqrt{4+v^2(4u^2+1)(4v+3)^2}$.
Since the signed area density function 
$\lambda$ is written as
$\lambda=v\hat\lambda$,
the $u$-axis gives the singular set 
$\{(u,0)\,;\,u\in \R\}$.
As every singular point $(u,0)$
is of the first kind,
$f$ is a cuspidal edge.
The limiting normal curvature 
$\kappa_\nu(u)$ is identically zero
along the $u$-axis, so the product curvature is too.
The Gaussian curvature $K$
is given by $K=-4(4v+3)(8v+3)/\hat\lambda^4$,
which satisfies $K_v(u,0)=-9$.
Hence, 
the corresponding singular points of $ds^2$
are intrinsic ${5/2}$-cuspidal edges,
although $f$ is a cuspidal edge.}
\end{example}

{Kossowski \cite{Kossowski} proved 
a realization theorem of Kossowski metrics 
which admit only singular points
satisfying $K\,dA\ne0$.
In \cite{HNUY},
a realization theorem of Kossowski metrics 
at an arbitrary singular point is proved.
In the following Fact \ref{fact:HNUY-B},
we introduce the realization theorem, 
which is a restricted version of \cite[Theorem B]{HNUY}
so that $\gamma(t)$ is chosen to be the singular curve
(for more details, see \cite{HNUY}).}

\begin{fact}[{cf.\ \cite[Theorem B]{HNUY}}]
\label{fact:HNUY-B}
Let $ds^2$ be a germ of an analytic Kossowski metric
on $(\R^2,0)$,
and let $\gamma(t)$ $(|t|<\varepsilon)$ 
be a singular curve passing through a singular point $0=\gamma(0)$.
Assume that $0$ is a type {\rm I} singular point of $ds^2$.
Then, for given analytic function-germs $\omega(t)$, $\tau(t)$ at $t=0$,
there exists an analytic frontal-germ 
$f=f_{\omega,\tau} : (\R^2,0)\to (\R^3,0)$
satisfying the following properties\/{\rm :}
\begin{enumerate}
\item $ds^2$ is the first fundamental form of $f$,
\item the limiting normal curvature function germ along 
the singular curve $\gamma$ coincides with $e^{\omega(t)}$ 
for a suitable choice of a unit normal vector field $\nu$, 
\item $\tau(t)$ gives the torsion function 
germ of $\hat \gamma(t):=f\circ \gamma(t)$.
\end{enumerate}
The possibilities for the congruence classes of such an 
$f$ are at most two.
Moreover, if $\tau$ vanishes identically
$($i.e $\hat \gamma$ is a planar curve$)$,
then the congruence class of $f$ is uniquely determined.
\end{fact}

Using Fact \ref{fact:HNUY-B} and
an argument similar to that of Theorem \ref{thm:isom-defo}, 
we have the following realization theorem of 
Kossowski metrics with intrinsic ${5/2}$-cuspidal edges
with prescribed singular images,
which is an analogous to a result of 
\cite[Theorem 12]{NUY} for cuspidal edges
and \cite[Corollary D]{HNUY} for cuspidal cross caps.

\begin{theorem}\label{thm:isom-realization}
Let $ds^2$ be a germ of an analytic Kossowski metric
on $(\R^2,0)$.
Assume that $0$ is an intrinsic\/ ${5/2}$-cuspidal edge.
Take a germ of an analytic regular space curve\/ $\sigma(t)$
such that its curvature function\/ $\kappa(t)$ satisfies
$$
  \kappa > |\kappa_s|
$$
at\/ $0$,
where\/ $\kappa_s$ is the singular curvature of\/ $ds^2$ 
along the singular curve\/ $\gamma$.
Then there exists a germ of an analytic\/ ${5/2}$-cuspidal edge 
$f_{\sigma} : (\R^2,0)\to (\R^3,0)$
with non-vanishing limiting normal curvature
such that
\begin{enumerate}
\item
the first fundamental form of\/ $f_{\sigma}$ coincides with\/ $ds^2$,
\item 
the singular image\/ $f_{\sigma}\circ \gamma$ coincides with\/ $\sigma$.
\end{enumerate}
The possibilities for congruence classes of such an\/ $f_{\sigma}$ 
are at most two unless\/ $\tau$ vanishes identically.
On the other hand, if\/ $\tau$ vanishes identically\/
$($i.e., $\sigma$ is a planar curve\/$)$,
then the congruence class of\/ $f_{\sigma}$ is uniquely determined.
\end{theorem}

\begin{proof}
{Set $\omega(t)$ to be
$\omega(t) := \log \sqrt{\kappa(t)^2 - \kappa_s(t)^2}.$
Let $\tau(t)$ be the torsion function of $\sigma(t)$.
By Fact \ref{fact:HNUY-B}, 
there exists an analytic frontal-germ 
$f_{\sigma}:=f_{\omega,\tau}:(\R^2,0)\to(\R^3,0)$
such that the items (1)--(3) in Fact \ref{fact:HNUY-B} hold.
Thus, it suffices to show that
$f_{\sigma}$ has a ${5/2}$-cuspidal edge at $0$.
Since
the first fundamental form of $f_{\sigma}$ coincides with $ds^2$, 
the product curvature $\kappa_\Pi$ 
and the secondary product curvature $r_\Pi$
of $f_{\sigma}$ coincide with
those of $ds^2$, respectively.
Therefore, by Corollary \ref{cor:intrinsic-criterion},
we have that $f_{\sigma}:(\R^2,0)\to(\R^3,0)$ 
has a ${5/2}$-cuspidal edge at $0$.}
\end{proof}

\begin{remark}
{
We may suppose that $\sigma(t)$ is defined for $|t|<\varepsilon$. 
By Theorem \ref{thm:isom-realization}, 
there exists a frontal $f_\pm:(\R^2,0)\to (\R^3,0)$ 
having a $5/2$-cuspidal edge at $p=\sigma(0)$ 
such that $f_-$ is isometric to $f_+$ 
and $\sigma(t)=f_\pm\circ \gamma(t)$. 
On the other hand, reversing the orientation of $\sigma(t)$,
there exists a frontal $g_\pm:(\R^2,0)\to (\R^3,0)$ 
having a $5/2$-cuspidal edge at $p=\sigma(0)$ 
such that $g_-$ is isometric to $g_+$ 
and $\sigma(-t)=g\circ \gamma(-t)$. 
Thus if $\sigma$ is not planar,
there are totally four distinct $5/2$-cuspidal edges
$f_+,f_-,g_+$ and $g_-$
with the common first fundamental form whose
image of the singular curve coincides with 
$\sigma((-\varepsilon,\varepsilon))$ in general, 
see \cite{HNSUY} for details.}
\end{remark}

\appendix

\section{Proofs {of propositions}}\label{sec:coord}
\subsection{Proof of Proposition\/ {\rm \ref{prop:normalform}}}
We show the following 
proposition,
which is a normal form
of a singular point of the first kind.
\begin{proposition}\label{prop:normalformf}
Let\/ $f:(\R^2,0)\to(\R^3,0)$ be a frontal and\/ $0$ a singular
point of the first kind.
Then there exist a coordinate system\/ $(u,v)$ and 
an isometry\/ {$A$}
of\/ $\R^3$ such that
$$
A\circ f(u,v)
=(u,a_2(u)+v^2/2,a_3(u)+v^2 b_3(u,v))
$$
for some functions\/ $a_2,a_3,b_3$.
If\/ $0$ is a\/ ${5/2}$-cuspidal edge,
$b_3$ has the form\/
$b_3=c_3(u)+v^2c_4(u)+v^3c_5(u,v)$ for
some functions\/ $c_3,c_4,c_5$.
\end{proposition}
\begin{proof}
Let $\nu$ be a unit
normal vector field along  $f$.
Since $\rank df_0=1$, by an isometry $A$ on $\R^3$,
we may assume $df_0(X)=(*,0,0)$ for any $X\in T_0\R^2$
and $\nu(0,0)=(0,0,1)$,
where $*$ stands for a real number.
Since $0$ is a singular point of the first kind,
$S(f)$ is a regular curve in $(\R^2,0)$, and
$\eta$ is transversal to $S(f)$.
Thus there exists a coordinate system $(\bar u,\bar v)$
satisfying $S(f)=\{\bar v=0\}$ and $\eta={\partial}_{\bar v}$.
Since $f_{\bar u}(0,0)=(a,0,0)$ $(a\ne0)$, 
setting $u=f_1(\bar u,\bar v),v=\bar v$, 
the coordinate system
$(u,v)$ 
satisfies 
\begin{multline}\label{eq:normal100}
f(u,v)
=(u,f_2(u,v),f_3(u,v)),\\
(f_2)_u(0,0)=(f_3)_u(0,0)=0,\quad
\nu(0,0)=(0,0,1),
\end{multline}
where $f_1(u,v)$ is the first component of $f$.
Since $f_v(u,0)=0$,
there exist functions $a_2,a_3,b_2,b_3$ such that
$f_i(u,v)=a_i(u)+v^2b_i(u,v)/2$ $(i=2,3)$.
Since $0$ is non-degenerate,
$\lambda_v(0,0)\ne0$. Thus $\det(f_u,f_{vv},\nu)(0,0)=b_2(0,0)\ne0$.
Setting $\tilde u=u,\tilde v=v \sqrt{|b_2(u,v)|}$, 
\eqref{eq:normal100} is
$$
f(u,\tilde v)
=(u,a_2(u)\pm\tilde v^2/2,a_3(u)+\tilde v^2 \tilde b_3(u,\tilde v)).
$$
This shows the first assertion.

If $0$ is a ${5/2}$-cuspidal edge,
then by Lemma \ref{prop:criteria}, $\det(f_u,f_{vv},f_{vvv})(u,0)=0$ holds.
Thus we have the second assertion.
\end{proof}
Proposition \ref{prop:normalform} is now obvious
by Proposition \ref{prop:normalformf}.

\subsection{Proof of Proposition\/ {\rm \ref{prop:criteria}}}
\begin{proof}[Proof of Proposition\/ {\rm \ref{prop:criteria}}]
To show Proposition \ref{prop:criteria},
{firstly we show the independence of the condition
on the choice of the vector fields.
Obviously, the condition $(1)$ 
does not depend on
the choice of the vector fields.
Since the condition $(2)$ is equivalent to
$f$ not being a front (Fact \ref{fact:front-condition}), 
the condition $(2)$ does not depend on
the choice of the vector fields.
Moreover, by the proof of Lemma \ref{lem:welldef},
we see the independence of the condition $(3)$
on the choice of the vector fields.}

By Proposition \ref{prop:normalformf},
we may assume that $f$ {is written in} the form
$
f(u,v)
=(u,v^2,v^5c_5(u,v))
$.
There exist functions $c_6,c_7$ such that
$c_5(u,v)=c_6(u,v^2)+v c_7(u,v^2)$.
{Considering $\Phi_1\circ f(u,v)$, where
$\Phi_1(X,Y,Z)=(X,Y,Z-Y^3c_7(X,Y))$,}
we may assume that $f$ has the form
$
f(u,v)
=(u,v^2,v^5c_6(u,v^2))
$.
Then a pair of vector fields $\xi=\partial_u$, $\eta=\partial_v$
satisfies
the condition of Proposition \ref{prop:criteria}
with \eqref{eq:eta}, and we see that $l=0$.
By condition \ref{itm:25new} of Proposition \ref{prop:criteria},
we see $c_6(0,0)\ne0$.
We set ${\Phi_2}(X,Y,Z)=(X,Y,Z/c_6(X,Y))$.
Then ${\Phi_2}\circ f=(u,v^2,v^5)$, which shows the assertion.
\end{proof}

\begin{acknowledgements}
The authors would like to thank Wayne Rossman,
Masaaki Umehara and Kotaro Yamada 
for helpful comments.
{The authors also thank Yuki Matsui for helping with calculations,
and the referee for a careful reading and helpful comments.}
\end{acknowledgements}

\medskip

\begin{thebibliography}{20}
\bibitem{fukui}
T. Fukui,
{\it Local differential geometry of cuspidal edge and swallowtail},
preprint, 
\verb|http://www.rimath.saitama-u.ac.jp/lab.jp/Fukui/preprint/CE_ST.pdf|.
\bibitem{ft}
T. Fukunaga and M. Takahashi,
{\it Evolutes of fronts in the Euclidean plane},
J. Singul. {\bf 10} (2014), 92--107. 
\bibitem{intcr}
M. Hasegawa, A. Honda, K. Naokawa, M. Umehara and K. Yamada,
{\it Intrinsic invariants of cross caps},
Selecta Math. (N.S.) {\bf 20} (2014), no. 3, 769--785.
\bibitem{HHNSUY}
M. Hasegawa, A. Honda, K. Naokawa, K. Saji, M. Umehara and K. Yamada,
{\it Intrinsic properties of surfaces with singularities},
Internat. J. Math. {\bf 26} (2015), no. 4, 1540008, 34 pp.
\bibitem{Honda}
A. Honda,
{\it On associate families of spacelike Delaunay surfaces},
in Real and Complex Singularities, Contemporary Mathematics, vol.\ 675, 
Amer.\ Math.\ Soc., Providence, RI, 2016, pp.\ 103--120.
\bibitem{HKS}
A. Honda, M. Koiso and K. Saji,
{\it Fold singularities on spacelike 
CMC surfaces in Lorentz-Minkowski space},
Hokkaido Math. J. {\bf 47} (2018), no. 2, 245--267.
\bibitem{HNUY0}
A. Honda, K. Naokawa, M. Umehara and K. Yamada,
{\it Isometric realization of cross caps as formal 
power series and its applications},
to appear in Hokkaido Mathematical Journal, arXiv:1601.06265.
\bibitem{HNUY}
A. Honda, K. Naokawa, M. Umehara and K. Yamada,
{\it Isometric deformations of wave fronts at non-degenerate singular points},
Preprint, 2017, arXiv:1710.02999.
\bibitem{HNSUY}
{ A. Honda, K. Naokawa, K. Saji, M. Umehara and K. Yamada,
{\it The duality on generalized cuspidal edges preserving 
their images of singular sets and first fundamental forms},
in preparation.}
\bibitem{ishiyama}
G. Ishikawa and T. Yamashita,
{\it Singularities of tangent surfaces to generic space curves}
J. Geom. {\bf 108} (2017), 301--318.
\bibitem{ist}
S. Izumiya, K. Saji and N. Takeuchi,
{\it Flat surfaces along cuspidal edges},
J. Singul. {\bf 16} (2017), 73--100.
\bibitem{KRSUY}
 M. Kokubu,  W. Rossman, K. Saji, M. Umehara, and K. Yamada,
 {\it Singularities of flat fronts in hyperbolic $3$-space},
 Pacific J. Math. {\bf 221} (2005), 303--351.
\bibitem{Kossowski}
  M.~Kossowski,
  {\it Realizing a singular first fundamental form as 
  a nonimmersed surface in Euclidean 3-space},
  J.\ Geom.\ {\bf 81} (2004), 101--113.
\bibitem{MS}
 L. F. Martins and K. Saji,
 {\it Geometric invariants of cuspidal edges},
 Canad. J. Math. {\bf 68} (2016), no. 2, 445--462.
\bibitem{MSUY}
L. F. Martins, K. Saji, M. Umehara and K. Yamada,
{\itshape Behavior of Gaussian curvature and
mean curvature near non-degenerate singular
points on wave fronts},
Geometry and Topology of Manifold,
Springer Proc.\ in Math.\ \& Stat. {\bf 154},
2016, Springer, 247--282.
\bibitem{NUY}
K. Naokawa, M. Umehara and K. Yamada,
{\it Isometric deformations of cuspidal edges},
Tohoku Math. J. (2) {\bf 68} (2016), no. 1, 73--90. 
\bibitem{OSaji}
R. Oset Sinha, and K. Saji,
{\it On the geometry of folded cuspidal edges},
Rev.\ Mat.\ Complut.\ {\bf 31} (2018), 627--650.
\bibitem{ot}
R. Oset Sinha, and F. Tari,
{\it On the flat geometry of the cuspidal edge},
Osaka J.\ Math.\ {\bf 55} (2018), no. 3, 393--421.
\bibitem{porteous}
I. R. Porteous,
{\it Geometric differentiation. 
For the intelligence of curves and surfaces}. 
Second edition. Cambridge University Press, Cambridge, 2001. 
\bibitem{SUY}
 K. Saji, M. Umehara, and K. Yamada,
 {\it The geometry of fronts},
 Ann.\ of Math. {\bf 169} (2009), 491--529.
\bibitem{SUYcamb}
 K. Saji, M. Umehara, and K. Yamada,
 {\it $A\sb k$ singularities of wave fronts},
 Math.\ Proc.\ Camb. Phil. Soc. {\bf 146} (2009), 731--746.
\bibitem{SUY2}
 K. Saji,  M. Umehara, and K. Yamada,
 {\it The duality between singular points and
 inflection points on wave fronts},
 Osaka J. Math.\ {\bf 47} (2010), 591--607.
\bibitem{SUY3}
 K. Saji,  M. Umehara, and K. Yamada,
 {\it An index formula for a bundle homomorphism 
 of the tangent bundle into a vector bundle of the same rank, and its applications},
 J.\ Math.\ Soc.\ Japan {\bf 69} (2017), 417--457.
\bibitem{su}
S. Shiba and M. Umehara, {\it The behavior of curvature functions at
cusps and inflection points}, Differential Geom. Appl. {\bf 30}
(2012), no. 3, 285--299.
\bibitem{teramoto}
K. Teramoto, 
{\it Parallel and dual surfaces of cuspidal edges},
Differential Geom. Appl. {\bf 44} (2016), 52--62.
\bibitem{teramoto2}
K. Teramoto, 
{\it Principal curvatures and parallel surfaces of wave fronts}, 
to appear in Adv. Geom., arXiv:1612.00577.
\bibitem{u}
M. Umehara,
{\it
Differential geometry on surfaces with singularities},
in: H. Arai, T. Sunada, K. Ueno (Eds.),
The World of Singularities, Nippon-Hyoron-sha
Co., Ltd., 2005, pp. 50--64 (in Japanese).
\bibitem{UY_hokudai}
 M. Umehara and K. Yamada,
 {\it Maximal surfaces with singularities in Minkowski space},
 Hokkaido Math.\ J.\ {\bf 35} (2006), 13--40.
\end{thebibliography}
\end{document}